\documentclass[11pt, a4paper]{amsart}
\usepackage[utf8]{inputenc}
\usepackage[margin=0.9in]{geometry}

\usepackage{comment}
\usepackage{lmodern}
\usepackage{upgreek}
\usepackage{color}
\usepackage{amsmath}
\usepackage{amssymb}
\usepackage{amsthm}
\usepackage{tikz-cd}

\usepackage{pdfpages}

\usepackage{fancyhdr}

\usepackage{setspace}

\usepackage{hyperref}
\hypersetup{
	colorlinks=true,
	linkcolor=blue,
	filecolor=blue,      
	urlcolor=blue,
	citecolor=blue,
	pdftitle={Connectedness in Codimension One and the Non-S2 Locus},
	pdfauthor={Likun Xie},
	pdfsubject={Commutative algebra},
	pdfkeywords={connectedness in codimension one, non-S2 locus, S2-ification}
}

\numberwithin{equation}{section}
\newtheorem{theorem}{Theorem}[section]
\newtheorem{lemma}[theorem]{Lemma}
\newtheorem{proposition}[theorem]{Proposition}
\newtheorem{corollary}[theorem]{Corollary}

\theoremstyle{remark}
\newtheorem{remark}[theorem]{Remark}

\theoremstyle{definition}

\newtheorem*{prob*}{Problem}

\newcommand{\Spec}{\operatorname{Spec}}
\newcommand{\Supp}{\operatorname{Supp}}
\newcommand{\Ass}{\operatorname{Ass}}
\newcommand{\End}{\operatorname{End}}
\newcommand{\Assh}{\operatorname{Assh}}

\newcommand{\ann}{\operatorname{ann}}
\newcommand{\Hom}{\operatorname{Hom}}
\newcommand{\Ext}{\operatorname{Ext}}
\newcommand{\depth}{\operatorname{depth}}

\title[Connectedness in Codimension One and the Non-\(S_2\) Locus]
{Connectedness in Codimension One and the Non-\(S_2\) Locus}

\author{Likun Xie}
\address{Max-Planck-Institut für Mathematik
	Vivatsgasse 7, 53111, Bonn, Germany} 
\email{xie@mpim-bonn.mpg.de}

\begin{document}

\begin{abstract}
	We formulate a structural principle for finite  \(S_2\)-objects:
coherent
	\(S_2\)-sheaves  and finitely generated graded \(S_2\)-modules   decompose canonically according to the connected
	components in codimension \(1\) of their support. This gives
	criteria relating indecomposability of \(S_2\)-objects to
	connectedness in codimension $1$ of their supports, and extends the Hochster--Huneke  correspondences  for complete local rings between connectedness in codimension  \(1\), indecomposability of canonical modules,  and localness of the
	\(S_2\)-ifications.

As a consequence, if \(A\) is a local ring admitting a canonical module
\(\omega_A\), there are canonical decompositions of both \(\omega_A\) and the
\(S_2\)-ification \(\operatorname{End}_A(\omega_A)\) whose indecomposable summands are the
canonical modules and \(S_2\)-ifications of the quotient rings associated to the
connected components in codimension \(1\). 
We then apply this viewpoint to the non-\(S_2\) locus. For $A$ equidimensional and
unmixed, this locus is naturally
realized as \(\operatorname{Supp}_A C\) via the \(S_2\)-ification sequence
\(0 \to A \to \operatorname{End}_A(\omega_A) \to C \to 0 \). 
  The natural map between deficiency modules
\(K^{\dim C+1}(A)\to K^{\dim C}(C)\)
identifies the canonical module  \(K^{\dim C}(C)\) with the \(S_2\)-hull of \(K^{\dim C+1}(A)\). Under suitable
conditions, this allows  
codimension-\(1\) connectedness of the non-\(S_2\) locus to be detected by the
deficiency module \(K^{\dim C+1}(A)\). We illustrate the theory with examples
and apply it to codimension \(2\) lattice ideals, obtaining
connectedness-in-codimension-\(1\) results for the non-\(S_2\) loci of certain    toric and lattice rings.
\end{abstract}

	\maketitle	
	\section{Introduction}
Connectedness in codimension \(1\) plays a central role in local algebra.    Faltings' connectedness theorem \cite[Thm.~6]{falting},  in local form,
says that if \((A,\mathfrak m)\) is a complete local domain and \(\mathfrak a\)
is generated by at most \(\dim A-2\) elements, then the punctured spectrum
\(\Spec(A/\mathfrak a)\setminus\{\mathfrak m/\mathfrak a\}\)
is connected.
Hochster and Huneke \cite{hochster} generalized this result by replacing the
domain hypothesis with the indecomposability of 
\(H^{\dim A}_{\mathfrak m}(A)\), equivalently of the canonical module \(\omega_A\).
They further related this condition to the localness of the \(S_2\)-ification
\(\End_A(\omega_A)\) and to connectedness in codimension \(1\) of the
top-dimensional part of \(\Spec A\); see \cite[Thm.~3.3, Thm.~3.6]{hochster}. The generalized Faltings
connectedness theorem can also be viewed as a consequence of codimension-\(1\) connectedness together with
Grothendieck's connectedness theorem. Indeed, in the complete equidimensional
case, connectedness in codimension \(1\) gives \(c(A)\ge \dim A-1\). Hence
Grothendieck's connectedness theorem \cite[Thm.~19.2.10]{local cohomology}  yields
\[
c(A/\mathfrak a)
\ge
\min\{c(A),\operatorname{sdim}A-1\}
-\operatorname{ara}(\mathfrak a)\ge  \dim A-1-(\dim A-2)=1,
\]
so
\(\Spec(A/\mathfrak a)\setminus\{\mathfrak m/\mathfrak a\}\)
is connected; see Section~\ref{connectedness_section}.

Schenzel \cite{connectedness} further recast these connectedness phenomena in
terms of ideal transforms and local cohomology, and obtained additional
equivalences between connectedness properties and indecomposability of certain
modules. A common principle behind these results is that one detects
connectedness of a space by finding a module, or more generally a sheaf of
sections, whose decomposition reflects the connected components of that space. We briefly discuss this section-functor viewpoint in
Remark~\ref{remark:section_functor_connectedness}.
 
The theme of this paper is that the Hochster--Huneke equivalences between
connectedness in codimension \(1\), indecomposability of \(\omega_A\), and
localness of the \(S_2\)-ification are instances of a more general principle:
finite \(S_2\)-objects decompose canonically according to the connected
components in codimension \(1\) of their support. This viewpoint leads
naturally to structure theorems for coherent \(S_2\)-sheaves, finitely
generated graded \(S_2\)-modules, canonical modules, and \(S_2\)-ifications.
It also provides a way to study codimension-$1$ connectedness of defect loci
through naturally associated \(S_2\)-modules, especially when the defect locus
is more naturally described as the support of a defect module than by an
explicit defining ring.

We begin with the following indecomposability criterion for coherent
\(S_2\)-sheaves.
	\begin{theorem} 
	Let \(X\) be a Noetherian scheme and let \(\mathcal F\) be a nonzero coherent
\(S_2\)-sheaf on \(X\). If \(\mathcal F\) is indecomposable, then
\(\operatorname{Supp}\mathcal F\) is connected in codimension \(1\).

Furthermore, assume that for every \(x\in \operatorname{Supp}\mathcal F\) with
\(	\operatorname{codim}_{\operatorname{Supp}\mathcal F}(x)\le 1\),
the stalk \(\mathcal F_x\) is indecomposable as an \(\mathcal O_{X,x}\)-module.
Then
\[
\mathcal F \text{ is indecomposable}
\quad\Longleftrightarrow\quad
\operatorname{Supp}\mathcal F \text{ is connected in codimension }1.
\]
	\end{theorem}
 
More generally, coherent \(S_2\)-sheaves admit canonical decompositions
according to the connected components in codimension \(1\) of their support.
	\begin{theorem} \label{intro_structure}
		Let \(X\) be a Noetherian scheme, and let \(\mathcal F\) be a nonzero coherent
	\(S_2\)-sheaf on \(X\). Let
	\[
	\operatorname{Supp}\mathcal F=S_1\cup\cdots\cup S_r
	\]
	be the decomposition of \(\operatorname{Supp}\mathcal F\) into its connected
	components in codimension \(1\). Then there is a decomposition
	\[
	\mathcal F\cong \bigoplus_{i=1}^r \mathcal F^{(i)},\qquad  \operatorname{Supp}\mathcal F^{(i)}=S_i\quad \text{for all } i,
	\]
	into coherent \(S_2\)-sheaves.
	This decomposition is canonical up to reordering.
	
	Moreover, assume that \(\operatorname{Supp}\mathcal F\) is equidimensional and
	that \(\mathcal F_x\) is indecomposable for every
	\(x\in\operatorname{Supp}\mathcal F\) with
	\(\operatorname{codim}_{\operatorname{Supp}\mathcal F}(x)\le 1\).
	Then each \(\mathcal F^{(i)}\) is indecomposable. In particular, \(\mathcal F\)
	admits a decomposition into indecomposable summands which is unique up to
	reordering.
	\end{theorem}
We also prove graded analogues, Theorems~\ref{indecomposable_graded} and
\ref{structure_theorem_graded}, for finitely generated \(\Gamma\)-graded
\(S_2\)-modules over Noetherian \(\Gamma\)-graded rings, where \(\Gamma\) is a
torsion-free abelian group.

For a local ring \(A\) with canonical module \(\omega_A\), the ring
\(B:=\operatorname{End}_A(\omega_A)\)
is the \(S_2\)-ification of $A$, and
\(\operatorname{Supp}_A B=\operatorname{Spec}(A/J(A))\). In the
complete local case, \(B\) decomposes as a product of local rings, and this
induces a corresponding decomposition of \(\omega_A\); see
\cite[Thm.~3.2(4)]{aoyama} and \cite[(2.2\,k)]{hochster}. Sather-Wagstaff and Spiroff \cite{counting} showed by a combinatorial argument  that the number of local factors of \(B\) is precisely
the number of connected components in codimension \(1\) of \(\operatorname{Spec}(A/J(A))\).

Using the structure theorem for finite \(S_2\)-modules, we extend this picture
beyond the complete local case. Namely, the connected components in codimension
\(1\) of \(\operatorname{Spec}(A/J(A))\) canonically determine the decomposition
of both \(\omega_A\) and \(B\). Completeness is only needed to ensure that the
factors of \(B\) are local.
\begin{theorem} 
 Let $A$ be a local ring that is a homomorphic image of a Gorenstein local ring. Let $(0) = \bigcap_{p \in \Ass A} Q(p)$ be a primary decomposition of $0$ in $A$ so that $J(A) = \bigcap_{p \in \Assh A} Q(p)$.  Let
\[
\Spec(A/J(A))=S_1\cup\cdots\cup S_r
\]
be the decomposition into connected components in codimension \(1\).

For each \(i\), let \(T_i:=S_i\cap \Assh(A)\). Define
\[
J_i := \bigcap_{p \in T_i} Q(p), \qquad A_i := A/J_i.
\]
\begin{enumerate}
	\item[(a)] The canonical module $\omega_A$ admits a decomposition
	\[	\omega_A \cong \omega_1 \oplus \cdots \oplus \omega_r,\qquad \Supp_A(\omega_i)=S_i \quad \text{for all } i.\]
	Each $\omega_i$ is indecomposable, and this gives the unique decomposition of $\omega_A$ into indecomposable $A$-modules, up to permutation.
	
	\item[(b)] Let $B := \operatorname{End}_A(\omega_A)$. Then   
	\[
	B \cong \prod_{i=1}^r B_i, \qquad B_i := \operatorname{End}_A(\omega_i),\qquad \Supp_A(B_i)=S_i \quad \text{for all } i.
	\]
	Each \(B_i\) is connected, equivalently indecomposable as a ring, and is also
	indecomposable as an \(A\)-module. This is the unique decomposition of \(B\) as
	a product of connected rings, or as a direct sum of indecomposable
	\(A\)-modules, up to permutation.
	
	\item[(c)] Each $\omega_i$ is the canonical module of $A_i$,   and $B_i \cong \End_A(\omega_i)$ is the $S_2$-ification of $A_i$. Let $C$ be the cokernel of the natural injection
	\(\alpha \colon A/J(A) \to B,\)
	and let $C_i$ be the cokernel of the natural injection
	\(	\alpha_i \colon A_i \to B_i.\)
	Then $\alpha$ factors through $\bigoplus_i \alpha_i$, and we have an exact sequence
	\[
	0 \to A/J(A) \to \bigoplus_{i=1}^r A_i \to C \to \bigoplus_{i=1}^r C_i \to 0.
	\]
	\item[(d)] Each $\omega_i$ is the canonical module of $B_i$.
	
	\item[(e)] If $A$ is complete local, then each $B_i$ is local.
\end{enumerate}
\end{theorem}

The same \(S_2\)-module viewpoint is useful in situations where the relevant
closed subset is more naturally described by a defect module than by an explicit
ring. For example, suppose one has an exact sequence
\begin{equation}\label{defect_sequence}
		0 \to A \to B \to C \to 0,
\end{equation}
where \(B\) is an \(S_2\)-ification, a Macaulayfication, a reflexive hull, an
\(S_2\)-hull, or a normalization. Then \(C\) measures the corresponding defect,
and \(\Supp C\) describes the defect locus. In this setting, a natural finite
\(S_2\)-object reflecting the codimension-one geometry of the defect locus is the
canonical module \(K_C\) viewed as an $A$-module; see \eqref{def_canonical_module}.
	
	In this paper, we focus on the \(S_2\)-ification and the non-\(S_2\) locus. 
Assume $C\neq 0$, and let \(t = \dim C\). 
	From the \(S_2\)-ification sequence \eqref{defect_sequence}  there is a natural map from the next deficiency module of \(A\) to the canonical module of \(C\) and this map is in fact an $S_2$-hull map. 
	\begin{theorem} 
		Let $(A,\mathfrak m)$ be a local ring which is a homomorphic image of a Gorenstein local ring and  $J(A)=0$. Assume \(C\neq 0\), and let \(t=\dim C\).	
		Then the connecting morphism arising from the exact sequence \eqref{defect_sequence}
		\[
		\varphi \colon K^{t+1}(A) \to K^t(C)
		\]
		is the $S_2$-hull of $K^{t+1}(A)$ in the sense of \cite[Def.\ 9.3]{Kollar2023}. In particular,
		\[
		K^t(C)\cong K_{K_{K^{t+1}(A)}}.
		\]
	\end{theorem}

Under suitable hypotheses, this allows the connectedness in codimension \(1\) of $\Supp(C)$ to be detected by the next deficiency module of $A$.

As an application, we consider lattice ideals of codimension \(2\). Let
\(S=k[x_1,\dots,x_n]\) be a polynomial ring over a field \(k\), and let
\(\mathcal L\subseteq \mathbb Z^n\) be a sublattice of rank \(2\). Set
\(\Gamma:=\mathbb Z^n/\mathcal L\) and \(A:=S/I_{\mathcal L}\), so that
\(\dim A=d=n-2\). Using the \(3\)-step minimal resolution of
\(I_{\mathcal L}\) due to Peeva and Sturmfels \cite{peeva}, we show that the
deficiency module \(K^{d-1}(A)\) is \(\Gamma\)-graded indecomposable. Under
additional hypotheses, this allows a direct comparison with the canonical
module \(K^{d-2}(C)\) of the defect module \(C\), yielding connectedness in
codimension \(1\) of the top-dimensional part of the non-\(S_2\) locus of
\(A\). The condition that \(I_{\mathcal L}\) is minimally generated by at
least \(4\) elements is equivalent to \(A\) being non-Cohen--Macaulay.
\begin{theorem}
	Let $I_{\mathcal L}$ be a lattice ideal of codimension $2$ minimally  generated by at least $4$ elements, and set $A := S/I_{\mathcal L}$. Let $d := \dim A = n-2$. Then $K^{d-1}(A)$ is $\Gamma$-graded indecomposable.

If the non-\(S_2\) locus of \(A\) is nonempty, then it
has dimension \(d-2\). 
Moreover, suppose that $I_{\mathcal L}$ is prime and   that \(K^{d-1}(A)\) is equidimensional and
\(S_2\), then   the top-dimensional part of the non-\(S_2\)
locus,
\(	\operatorname{non}\text{-}S_2^{\mathrm{top}}(A)\),  is connected in codimension \(1\). 
\end{theorem}
In particular, this applies to non-Cohen--Macaulay toric rings \(A\) of
codimension \(2\) whose canonical module \(\omega_A\) is Cohen--Macaulay, or
equivalently whose \(S_2\)-ification \(B\) is Cohen--Macaulay
\cite[Prop.~2.2]{S_2-ification_universal}. For such rings, the non-\(S_2\)
locus coincides with the non-Cohen--Macaulay locus. This includes, for example,
simplicial affine semigroup rings \cite[Thm.~6.4]{ccm}.
\begin{corollary} 
Let \(I_{\mathcal L}\) be a toric ideal of codimension \(2\), minimally
generated by at least \(4\) elements, and set \(A:=S/I_{\mathcal L}\). Suppose
that  the canonical module \(\omega_A\) of \(A\) is  Cohen--Macaulay; for example, this holds when \(A\) is a simplicial affine semigroup ring. Then
the non-\(S_2\) locus of \(A\) coincides with the non-Cohen--Macaulay locus of
\(A\), and its top-dimensional part is connected in codimension \(1\).
\end{corollary}

The same connectedness-in-codimension-\(1\) result also applies in certain
cases where \(I_{\mathcal L}\) is not necessarily prime. In the non-prime
case, the graded indecomposability result need not apply directly, since the
grading group $\Gamma$ is not torsion free. However, the case where \(I_{\mathcal L}\)
is minimally generated by exactly four elements is especially explicit which  allows us to
prove directly that \(K^{d-1}(A)\) is indecomposable as an ordinary
\(A\)-module. The nongraded indecomposability--connectedness criterion
then gives the following connectedness in codimension $1$ result.
\begin{proposition} 
	Let $I_{\mathcal L}$ be a lattice ideal of
codimension \(2\), minimally generated by \(4\) elements. Suppose that $A := S/I_{\mathcal L}$ is equidimensional and unmixed. Then the non-\(S_2\) locus of \(A\) is nonempty and has dimension
\(\dim A-2\). Moreover,  $\operatorname{non}\text{-}S_2^{\mathrm{top}}(A)$ is connected in
codimension \(1\), and  the non-\(S_2\) locus of \(A\) coincides
with the non-Cohen--Macaulay locus.
\end{proposition}
Finally, we include examples showing that \(S_2\)-ifications can be
\(S_k\) but not \(S_{k+1}\), and that the defect module \(C\) need not arise
naturally as a quotient ring of \(A\).

\subsection{Notation}
Throughout the paper, ``local ring'' means Noetherian local ring. For an $A$-module $M$, we denote by $\Assh_A(M)$ the  set of associated primes \(\mathfrak p \in \Ass_A(M)\)  such that $\dim A/\mathfrak p = \dim M$. We write $J(M)$ for the largest submodule of $M$ of dimension strictly less than $\dim M$, and define the top-dimensional quotient of $M$ by $\overline{M} := M/J(M)$.

We write $\omega_A$ for the canonical module of a local ring $A$, and $K_M$ for the canonical module of an $A$-module $M$ as defined in \eqref{def_canonical_module}. We say that $M$ is equidimensional if all minimal primes of $M$ have dimension $\dim M$, and that $M$ is unmixed if it has no embedded primes, i.e.\ if all associated primes of $M$ are minimal.

For a local ring $(A,\mathfrak m,k)$, we denote by $(-)^\vee$ the Matlis dual functor $(-)^\vee := \Hom_A(-,E_A(k))$, where $E_A(k)$ is the injective hull of $k$ over $A$.

\section{The non-\(S_2\) locus and deficiency modules}
In this section, we briefly review the non-\(S_2\) locus, ideal transforms, canonical modules, and the deficiency modules of finitely generated modules.
\subsection{The non-\(S_2 \) locus}\label{set-up_non-S_2}
Let $(A,\mathfrak m)$ be a local ring that is a homomorphic image of a Gorenstein local ring $R$. 
Set $\dim A=d$ and $\dim R=n$, and let
\[
\omega_A \cong \operatorname{Ext}_R^{\,n-d}(A,R)
\]
be the canonical module of $A$.
Let  $M$ be a finitely generated $A$-module with $\dim M = k$. 
Define $J(M)$ to be the largest submodule of $M$ of dimension $< k$.
Concretely, if
\[0=\bigcap_{i=1}^n N_i\]
is a primary decomposition of the zero submodule of $M$, where each $N_i$ is $\mathfrak p_i$-primary, then
\[
J(M)=\bigcap_{\dim A/\mathfrak p_i=k} N_i.
\]
The \emph{top-dimensional quotient} of $M$ is
\begin{equation}\label{top}
	\overline{M}:=M/J(M).
\end{equation}
Then $\overline{M}$ is unmixed and equidimensional.

There is a natural homomorphism induced by multiplication
\[
A \longrightarrow B:=\Hom_A(\omega_A,\omega_A),
\]
whose kernel is precisely \(J(A)\); see \cite[Remark~12.2.5]{local cohomology}. Let \(C\) denote the cokernel of this map. After replacing \(A\) by its top-dimensional quotient \(A/J(A)\), we may assume that \(A\) is unmixed and equidimensional. Thus there is  a short exact sequence
\begin{equation}\label{exact_ABC}
	0 \longrightarrow A  \longrightarrow B \longrightarrow C \longrightarrow 0.
\end{equation}
Let $I:=\ann_A(C)$. Then $\operatorname{ht}(I)\ge 2$, and the non-\(S_2\) locus of $A$ is given by \cite[Thm.~12.2.7, Lem.~12.2.10]{local cohomology}
\[
\operatorname{non}\text{-}S_2(A)
:=\{\mathfrak p\in \Spec A \mid A_{\mathfrak p}\text{ is not }S_2\}
=V(I).
\]

\subsection{Ideal transform}

Let \(A\) be a Noetherian ring, \(\mathfrak a\subseteq A\) an ideal, and
\(M\) an \(A\)-module.  The \emph{ideal transform} of \(M\) with
respect to \(\mathfrak a\), also called the \(\mathfrak a\)-transform of \(M\), is defined by
\[
D_{\mathfrak a}(M) := \varinjlim_{n \in \mathbb N} \operatorname{Hom}_A(\mathfrak a^n,M).
\]
Set \(X=\Spec A\) and \(U=X\setminus V(\mathfrak a)\). By the affine Deligne isomorphism
\cite[Ex.~III.3.7]{ag_hartshorne}
(see also \cite[Thm.~20.1.14,  Example~20.3.4]{local cohomology}),
there is a natural isomorphism
\begin{equation}\label{delign}
	D_{\mathfrak a}(M)\cong \Gamma(U,\widetilde M).
\end{equation}
More generally, let $(\Lambda,\leq)$ be a directed partially ordered set, and let 
\(\mathfrak B = (\mathfrak b_\alpha)_{\alpha \in \Lambda}\)
be an inverse system of ideals of $A$ such that for every
$\alpha,\gamma\in \Lambda$ there exists $\delta\in \Lambda$ satisfying
\(\mathfrak b_\delta \subseteq \mathfrak b_\alpha \mathfrak b_\gamma .\)
Following \cite[Section~2.2]{local cohomology}, the \emph{generalized ideal transform}
of $M$ with respect to $\mathfrak B$ is defined by
\[
D_{\mathfrak B}(M)
:=
\varinjlim_{\alpha\in \Lambda}
\Hom_A(\mathfrak b_\alpha,M).
\]
Let $\mathfrak H$ denote the directed system of ideals of $A$ of height at least $2$,
partially ordered by reverse inclusion. Let $\mathfrak S$ denote the directed system of ideals
$\mathfrak a\subseteq A$ such that
\(V(\mathfrak a)\subseteq \operatorname{non}\text{-}S_2(A).\)
By \cite[Thm.~12.3.10(v)]{local cohomology} together with \cite[Lem.~6.2]{connectedness},
there are natural $A$-algebra isomorphisms
\[
B
\cong
D_I(A)
\cong
D_{\mathfrak H}(A)
\cong
D_{\mathfrak S}(A).
\]
Thus the \(S_2\)-ification may be viewed equivalently as an ideal transform.

\subsection{Canonical modules and deficiency modules}

The notion of canonical module extends naturally from rings to finitely generated modules.
Let \((A,\mathfrak m)\) be a local ring that is a homomorphic image of a Gorenstein local ring \(R\), and set \(\dim A=d\) and \(\dim R=n\).
For a finitely generated \(A\)-module \(M\) and an integer \(i\), define the \emph{\(i\)-th deficiency module of \(M\)} by
\[
K_A^i(M):=\Ext_R^{n-i}(M,R),
\]
so that \(K_A^d(A)=\omega_A\) is the canonical module of \(A\).
If \(\dim M=t\), the \emph{canonical module of \(M\)} is
\begin{equation}\label{def_canonical_module}
K_M:=K_A^t(M)=\Ext_R^{n-t}(M,R).
\end{equation}
When the base ring $A$ is fixed, we write $K^i(M)$ for $K^i_A(M)$. 
 
This agrees with Schenzel's definition via a dualizing complex; see
\cite{lecture_notes,canonical_module}. By \cite[Corollary~1.4]{dualizing},
a local ring admits a dualizing complex if and only if it is a homomorphic
image of a Gorenstein local ring. Thus, throughout this section, working with
local rings that are homomorphic images of Gorenstein local rings is equivalent
to working with local rings that admit dualizing complexes, and the two
definitions of canonical modules and deficiency modules agree.

We record some basic properties that will be used repeatedly; see \cite[Lem.~3.1.1 and 3.2.1]{lecture_notes} and \cite[Prop.~3.1]{canonical_module}.

\begin{proposition}\label{prop_canonical_module}
	Let $M$ be a finitely generated $A$-module. Then
	\begin{enumerate}
		\item[(a)] $\dim K_M = \dim M$, and $\dim K ^i(M) \leq i$ for all $0 \leq i < \dim M$.
		
		\item[(b)] $\operatorname{Ass}_A(K_M) = \operatorname{Assh}_A(M)$, and $K_M$ is $S_2$.
		
		\item[(c)] $M$ satisfies $ S_r  $ if and only if $\dim K ^i(M) \leq i - r$ for all $0 \leq i < \dim M$.
	\end{enumerate}
\end{proposition}

\section{Connectedness and decomposition theorems}

We recall the convention used throughout this section. Let \(X\) be a
Noetherian topological space. We say that \(X\) is \emph{connected in
	codimension \(k\)} if, for every closed subset \(Z\subseteq X\) with
\(\operatorname{codim}_X Z>k\), the complement \(X\setminus Z\) is connected.

This is equivalent to the following condition \cite[Prop.~1.1]{hartshorne}:
for any two irreducible components \(X'\) and \(Y'\) of \(X\), there is a chain
of irreducible components
\[
X'=X_0,X_1,\dots,X_r=Y'
\]
such that
\[
\operatorname{codim}_X(X_{i-1}\cap X_i)\le k
\qquad\text{for all }i=1,\dots,r.
\]

In particular, this chain condition defines an equivalence relation on the irreducible components of \(X\), partitioning them into disjoint equivalence classes. The union of the irreducible components in a given equivalence class is called a \emph{connected component of \(X\) in codimension \(k\)}.

If \(X=V(\mathfrak a)\subseteq \operatorname{Spec} A\), and \( X_i = V(\mathfrak p_i)\) are irreducible components of $X$, where $\mathfrak p_i$ are minimal primes over \(\mathfrak a\), then
\[
\operatorname{codim}_X(X_{i-1}\cap X_i)\le k
\]
is equivalent to the existence of a prime
\(\mathfrak P_i\in X\) such that
\[
\mathfrak p_{i-1}+\mathfrak p_i\subseteq \mathfrak P_i
\qquad\text{and}\qquad
\operatorname{ht}_{A/\mathfrak a}(\mathfrak P_i/\mathfrak a)\le k.
\]

\subsection{Indecomposability   and connectedness in codimension $1$}

  In this section, we prove
indecomposability--connectedness criteria for coherent \(S_2\)-sheaves and
for finitely generated \(\Gamma\)-graded \(S_2\)-modules, where \(\Gamma\) is a
torsion-free abelian group. The key input is the following extension property
for \(S_2\)-modules. A geometric proof can be found in
\cite[Thm.~1.17, Cor.~1.18]{schwede}; we include an algebraic proof for
completeness.

\begin{lemma}\cite[Cor.~1.18]{schwede}\label{lemma}
Let \(X\) be a Noetherian scheme and let \(\mathcal F\) be a coherent \(S_2\)-sheaf on \(X\). Let \(Y\subseteq X\) be a closed subset such that
\[
\operatorname{codim}_{\Supp\mathcal F}(Y\cap \Supp\mathcal F)\ge 2.
\]
Set \(U=X\setminus Y\), and let \(j\colon U\hookrightarrow X\) denote the inclusion. Then the natural map
\[
\mathcal F \to j_*(\mathcal F|_U)
\]
is an isomorphism.
\end{lemma}
\begin{proof}
	This is local on \(X\), so we may assume that \(X=\Spec A\) and
	\(\mathcal F=\widetilde M\), where \(M\) is a finitely generated \(S_2\)
	\(A\)-module. Write $U = \operatorname{Spec} A \setminus V(I)$ for some ideal $I \subseteq A$ such that $\operatorname{ht} \overline{I} \geq 2$, where $\overline{I}$ denotes the image of $I$ in $\overline{A} := A/\operatorname{ann} M$.
	
	Since \(M\) is \(S_2\) as an $A$-module, we have
	\[
\operatorname{depth}_I(M) = \inf_{\mathfrak{p} \supseteq I} \operatorname{depth} M_{\mathfrak{p}}
\geq \inf_{\mathfrak{p} \supseteq I} \min\{2,\dim M_{\mathfrak{p}}\}
= \min\{2, \inf_{\mathfrak{p} \supseteq \overline{I}} \dim \overline{A}_{\mathfrak{p}}\}
= \min\{2, \operatorname{ht} \overline{I}\} = 2.
\]
	Thus \(H_I^0(M)=H_I^1(M)=0\). By the isomorphism \eqref{delign} and the exact sequence for ideal transforms
	\cite[Thm.~2.2.6]{local cohomology},
	\[
	0\to H_I^0(M)\to M\to D_I(M)\to H_I^1(M)\to 0,
	\]
	we obtain $M \cong D_I(M) \cong \Gamma(U, \widetilde{M})$.  Since   
	\(  j_*(\widetilde M|_U)\cong   \widetilde{D_I(M)}\),
    it follows that
	\(	\widetilde M\cong j_*(\widetilde M|_U)\)
	as claimed.
\end{proof}

\begin{theorem}\label{sheaf_indecomposable}
	Let \(X\) be a Noetherian scheme and let \(\mathcal F\) be a nonzero coherent
	\(S_2\)-sheaf on \(X\). If \(\mathcal F\) is indecomposable, then
	\(\operatorname{Supp}\mathcal F\) is connected in codimension \(1\).
	
	Furthermore, assume that for every \(x\in \operatorname{Supp}\mathcal F\) with
\(	\operatorname{codim}_{\operatorname{Supp}\mathcal F}(x)\le 1\),
	the stalk \(\mathcal F_x\) is indecomposable as an \(\mathcal O_{X,x}\)-module.
	Then
	\[
	\mathcal F \text{ is indecomposable}
	\quad\Longleftrightarrow\quad
	\operatorname{Supp}\mathcal F \text{ is connected in codimension }1.
	\]
\end{theorem}
\begin{proof}
	Replacing \(X\) by the closed subscheme defined by
	\(\operatorname{ann}\mathcal F\), we may assume that
	\(\Supp\mathcal F=X\).
Suppose first that \(\mathcal F\) is indecomposable. If \(X\) is not connected
in codimension \(1\), then there exists a closed subset \(Y\subseteq X\) with
\(\operatorname{codim}_X(Y)\ge 2\) such that
\(U:=X\setminus Y\)
is disconnected. Write
\[
U=U_1\sqcup U_2
\]
with \(U_1,U_2\) nonempty open and closed subsets of \(U\). Then
\[
\mathcal F|_U
\cong
\mathcal F|_{U_1}\oplus \mathcal F|_{U_2}.
\]
By Lemma~\ref{lemma},
\[
\mathcal{F} \cong j_* (\mathcal{F}|_U) \cong j_*(\mathcal{F}|_{U_1})
\oplus j_*(\mathcal{F}|_{U_2}),
\]
where \(j:U\hookrightarrow X\) is the inclusion. Since \(U_i\neq\varnothing\)
and \(U_i\subseteq \Supp\mathcal F\), both summands are nonzero. This
contradicts the indecomposability of \(\mathcal F\). Hence \(X\) is connected
in codimension \(1\).

Conversely, suppose that \(X=\Supp\mathcal F\) is connected in codimension
\(1\), and that \(\mathcal F_x\) is indecomposable for every point \(x\in X\)
with \(\operatorname{codim}_X(x)\le 1\).  

 Assume that
\[
\mathcal{F} = \mathcal{F}_1 \oplus \mathcal{F}_2
\]
with $\mathcal{F}_1, \mathcal{F}_2 \neq 0$. Let $S_i := \operatorname{Supp}\mathcal{F}_i$, then
\[
X=S_1\cup S_2.
\]
Since \(\mathcal F_x\) is indecomposable at every
point of codimension at most \(1\), the intersection
\[Y:=S_1\cap S_2\]
has codimension at least \(2\) in \(X\). Indeed, if \(x\in S_1\cap S_2\), then
both \((\mathcal F_1)_x\) and \((\mathcal F_2)_x\) are nonzero, so
\(\mathcal F_x\) is decomposable.

We claim that both \(S_1\setminus Y\) and \(S_2\setminus Y\) are nonempty.
For instance,   if $S_1 \setminus Y = \varnothing$, then $S_1 \subseteq S_2$, hence $Y = S_1$, which has codimension at least $2$ in \(X\). Applying Lemma~\ref{lemma} to \(\mathcal F\) gives 
\(\mathcal F \cong j_*(\mathcal F|_{X\setminus S_1}).\)
But the inclusion \(\mathcal F_1\hookrightarrow \mathcal F\) restricts to zero on
\(X\setminus S_1\). Under the above isomorphism, this forces
\(\mathcal F_1=0\), a contradiction.
Thus \(S_1\setminus Y\neq\varnothing\),
and similarly \(S_2\setminus Y\neq\varnothing\).

Therefore
\(X\setminus Y=(S_1\setminus Y)\sqcup(S_2\setminus Y)\) is disconnected. 
  Since
\(\operatorname{codim}_X(Y)\ge 2\), this contradicts the assumption that
\(X\) is connected in codimension \(1\). Hence \(\mathcal F\) is
indecomposable.
\end{proof}
We also formulate a graded version. In the usual nonnegatively
$\mathbb Z$-graded setting one may pass to \(\operatorname{Proj} A\), but the
associated sheaf on \(\operatorname{Proj} A\) does not see the part of a graded module
supported in \(V(A_+)\), where \(A_+\) is the irrelevant ideal.   Since we want a statement about indecomposability of
the graded module itself, we work directly on \(\Spec A\).

Let \(A=\bigoplus_{\gamma\in \Gamma}A_\gamma\) be a Noetherian
\(\Gamma\)-graded ring, also called a graded ring of type \(\Gamma\) in
\cite[A.I.1]{graded}, where \(\Gamma\) is a torsion-free abelian group. We call a \(\Gamma\)-graded \(A\)-submodule of \(A\) a \emph{homogeneous
	ideal}, also called a \emph{graded ideal}. The
torsion-free hypothesis ensures that associated primes of     graded modules are homogeneous, a fact used in the proof below. Without this
hypothesis, the assertion can fail: for example, if $A=k[x]$ is given the
\(\mathbb Z/2\mathbb Z\)-grading with \(\deg x=1\), then the ideal $ (x^2-1)$ is homogeneous, but its minimal primes $(x-1)$ and $(x+1)$ are not homogeneous. For background on graded rings and modules, see \cite{graded}.

We say that a closed subset of \(\Spec A\) is \textit{\(\Gamma\)-homogeneous} if it can be defined by a homogeneous ideal. A finitely generated \(\Gamma\)-graded \(A\)-module
\(M\) is called \textit{\(\Gamma\)-graded indecomposable} if it is indecomposable in the
category of \(\Gamma\)-graded \(A\)-modules with degree-zero homomorphisms; see \cite[A.1]{graded}.

\begin{theorem}\label{indecomposable_graded}
Let $A=\bigoplus_{\gamma\in \Gamma} A_\gamma$ be a Noetherian $\Gamma$-graded ring, where $\Gamma$ is a torsion-free abelian group, and let $M$ be a nonzero finitely generated $\Gamma$-graded $A$-module satisfying \(S_2\). If $M$ is $\Gamma$-graded indecomposable, then
\(\Supp (M)\) is connected in codimension \(1\).

Moreover, assume that for every homogeneous prime $\mathfrak p \in \operatorname{Supp}(M)$ with $\operatorname{codim}_{\operatorname{Supp}(M)}(\mathfrak p) \le 1$, the homogeneous localization $M_{(\mathfrak p)}$ is $\Gamma$-graded indecomposable. Then \[ M \text{ is $\Gamma$-graded indecomposable } \Longleftrightarrow \operatorname{Supp}(M) \text{ is connected in codimension } 1. \]
\end{theorem}
\begin{proof}
	Suppose first that \(M\) is \(\Gamma\)-graded indecomposable. Set
	\(S=\operatorname{Supp}(M)\), and assume that \(S\) is not connected in
	codimension \(1\). Then there exists a closed subset $Y \subset S$ with $\operatorname{codim}_S Y \ge 2$ such that $U := S \setminus Y$ is disconnected. Write
	\[
	U=U_1\sqcup U_2
	\]
	with \(U_1\) and \(U_2\) nonempty open and closed subsets of \(U\).
	
Let $Z_i := \overline{U_i}^{\,S}$  for $i=1,2$. Since \(U_i\) is closed in \(U\), we have
\[
Z_i\cap U
=
\overline{U_i}^{\,S}\cap U
=
\overline{U_i}^{\,U}
=
U_i.
\]
Thus \((Z_1\cap Z_2)\cap U=\emptyset\), and hence
\[
Z_1\cap Z_2\subseteq Y.
\]

Since \(\operatorname{codim}_S Y\ge 2\),   \(Y\) contains no minimal
prime of \(S\). Hence every minimal prime of \(S\) lies in exactly one of
\(U_1\) and \(U_2\). It follows that each \(Z_i\) is a union of irreducible
components of \(S\), and that \(S=Z_1\cup Z_2\).
The minimal primes of \(S=\operatorname{Supp}(M)\) are associated primes
of \(M\) which are homogeneous by \cite[Thm.~II.7.3]{graded}. Therefore \(Z_1\) and \(Z_2\) are
\(\Gamma\)-homogeneous closed subsets of \(S\).

Set $Y' := Z_1 \cap Z_2$. Then \(Y'\) is a \(\Gamma\)-homogeneous closed subset of \(S\), and $\operatorname{codim}_S Y' \ge 2$. Moreover
\[
S\setminus Y'
=
(Z_1\setminus Y')\sqcup (Z_2\setminus Y')
\]
is a disjoint union of two nonempty open and closed subsets.

Let \(j\colon S\setminus Y'\hookrightarrow S\)  
denote the inclusion.    By the
\(S_2\)-extension property, Lemma~\ref{lemma}, we have
\[
\widetilde{M} \cong j_*\bigl(\widetilde{M}|_{S \setminus Y'}\bigr)
\cong
j_*\bigl(\widetilde{M}|_{Z_1 \setminus Y'}\bigr)
\oplus
j_*\bigl(\widetilde{M}|_{Z_2 \setminus Y'}\bigr).
\]

Since \(Z_2\) is homogeneous closed in \(S\), the open subset
\(Z_1\setminus Y'=S\setminus Z_2\) may be written as
 $ S  \setminus V(I)$ with $I=(f_1,\dots,f_r)$ generated by homogeneous elements. Then
\[
\Gamma(Z_1 \setminus Y', \widetilde{M})
=
\ker\!\Bigl(\bigoplus_i M_{f_i}
\rightarrow
\bigoplus_{i,j} M_{f_i f_j}\Bigr).
\]
Since the \(f_i\) are homogeneous, the localizations \(M_{f_i}\) and
\(M_{f_i f_j}\) are naturally \(\Gamma\)-graded and the localization maps
are degree-preserving. Hence
\(\Gamma(Z_1\setminus Y',\widetilde M)\) is a \(\Gamma\)-graded
\(A\)-module (see \cite[p.~4]{graded}). The same argument applies to
\(\Gamma(Z_2\setminus Y',\widetilde M)\).  
Taking global sections,
\[
M \cong \Gamma(Z_1 \setminus Y', \widetilde{M})
\oplus
\Gamma(Z_2 \setminus Y', \widetilde{M})
\]
is a decomposition of $M$ into two nonzero $\Gamma$-graded $A$-modules, a contradiction.

Conversely, assume that \(S=\operatorname{Supp}(M)\) is connected in
codimension \(1\), and assume that for every homogeneous prime
\(\mathfrak p\in S\) with
\(\operatorname{codim}_S(\mathfrak p)\le 1\), the homogeneous localization
\(M_{(\mathfrak p)}\) is \(\Gamma\)-graded indecomposable.

Suppose that $M \cong M_1 \oplus M_2$ is a nontrivial $\Gamma$-graded  decomposition. Then each $M_i$ is $\Gamma$-graded, and by \cite[Thm.~II.7.3]{graded},   \(T_i:=\operatorname{Supp}(M_i)\) is a homogeneous closed subset of \(S\) and $S = T_1 \cup T_2$.

If $\operatorname{codim}_S(T_1 \cap T_2) \le 1$, then  since $T_1 \cap T_2$ is homogeneous closed, by \cite[Thm.~II.7.3]{graded}  there exists a homogeneous  prime $\mathfrak p \in T_1 \cap T_2$ with $\operatorname{codim}_S(\mathfrak p) \le 1$. Then
\(M_{(\mathfrak p)} \cong (M_1)_{(\mathfrak p)} \oplus (M_2)_{(\mathfrak p)}\)
is a nontrivial $\Gamma$-graded decomposition, contradicting the hypothesis.
Therefore $\operatorname{codim}_S(T_1 \cap T_2) \ge 2$. Set
\(Y:=T_1\cap T_2\). Since \(M\) is \(S_2\), it has no
embedded  primes. Hence, \(\Ass_A(M)=\Ass_A(M_1)\cup\Ass_A(M_2)=\operatorname{Min}_A(M).\) Since \(M_i\neq 0\), each \(T_i \) must contain some
minimal prime in \(S \). As \(\operatorname{codim}_S Y\ge 2\), no
minimal prime of \(S\) lies in \(Y\). Thus   
\(T_i\setminus Y\neq\emptyset\) for \(i=1,2\). Therefore
\[
S\setminus Y=(T_1\setminus Y)\sqcup(T_2\setminus Y)
\]
is disconnected, contradicting the connectedness of \(S\) in codimension \(1\).
\end{proof}

We now turn from connectedness to decomposition.  The next result gives the corresponding decomposition
statement: a coherent \(S_2\)-sheaf decomposes canonically according to the
connected components in codimension \(1\) of its support.
\begin{theorem}[Structure theorem for coherent \(S_2\)-sheaves]
	\label{structure_S_2 sheaf}
	Let \(X\) be a Noetherian scheme, and let \(\mathcal F\) be a nonzero coherent
	\(S_2\)-sheaf on \(X\). Let
	\[
	\operatorname{Supp}\mathcal F=S_1\cup\cdots\cup S_r
	\]
	be the decomposition of \(\operatorname{Supp}\mathcal F\) into its connected
	components in codimension \(1\). Then there is a decomposition
	\[
	\mathcal F\cong \bigoplus_{i=1}^r \mathcal F^{(i)},\qquad  \operatorname{Supp}\mathcal F^{(i)}=S_i\quad \text{for all } i,
	\]
	into coherent \(S_2\)-sheaves.
	This decomposition is canonical up to reordering.
	
Moreover, assume that \(\operatorname{Supp}\mathcal F\) is equidimensional and
that \(\mathcal F_x\) is indecomposable for every
\(x\in\operatorname{Supp}\mathcal F\) with
\(\operatorname{codim}_{\operatorname{Supp}\mathcal F}(x)\le 1\).
Then each \(\mathcal F^{(i)}\) is indecomposable. In particular, \(\mathcal F\)
admits a decomposition into indecomposable summands which is unique up to
reordering.
\end{theorem}
\begin{proof}
	Replacing \(X\) by the closed subscheme defined by
	\(\operatorname{ann} \mathcal F\), we may assume that
	\(\operatorname{Supp}\mathcal F=X\).
	Since distinct  connected components in codimension $1$  meet in codimension at least $2$, the set
	\[
	Y := \bigcup_{i \neq j} (S_i \cap S_j)
	\]
	is a closed subset of codimension at least $2$ in   \(X\).  
	
	Set $U:=X\setminus Y$. Then
	\[
	U = U_1 \sqcup \cdots \sqcup U_r, \qquad U_i := S_i \setminus Y,
	\]
	where each $U_i$ is open and closed in $U$.   Hence
	\[
	\mathcal{F}|_U \cong \bigoplus_{i=1}^r  \mathcal F|_{U_i} .
	\]
	Let $j \colon U \hookrightarrow X$ and $j_i: U_i\hookrightarrow X$ denote the inclusions and define
	\(\mathcal{F}^{(i)} := (j_i)_*\bigl(\mathcal{F}|_{U_i}\bigr).\)
    By the \(S_2\)-extension property, Lemma~\ref{lemma}, we obtain
	\[
	\mathcal{F} \cong j_*(\mathcal{F}|_U) \cong \bigoplus_{i=1}^r \mathcal{F}^{(i)}.
	\]
	Moreover, $\operatorname{Supp}\mathcal{F}^{(i)} = S_i$, and each \(\mathcal{F}^{(i)}\) is coherent and \(S_2\), as it
	is a direct summand of \(\mathcal F\). 
Since each \(U_i\) is determined by \(S_i\), the summands
\(\mathcal F^{(i)}\) are uniquely determined by \(\mathcal F\) and the
  components \(S_i\). Thus the decomposition is canonical up
to reordering.

Now assume that \(\operatorname{Supp}\mathcal F=X\) is equidimensional and that
\(\mathcal F_x\) is indecomposable for every \(x\in X\) with
\(\operatorname{codim}_X(x)\le 1\). Let \(x\in S_i\) satisfy
\(\operatorname{codim}_{S_i}(x)\le 1\). Since \(X\) is equidimensional, we have
\(\operatorname{codim}_X(x)=\operatorname{codim}_{S_i}(x)\le 1.\)
It follows that  \(\mathcal F^{(i)}_x\cong \mathcal F_x\) is indecomposable. 
Thus \(\mathcal F^{(i)}\) satisfies the codimension-\(1\) local
indecomposability hypothesis on \(S_i\).
 Since \(S_i\) is connected in
codimension \(1\), Theorem~\ref{sheaf_indecomposable} implies that
\(\mathcal F^{(i)}\) is indecomposable.

It remains to show that this decomposition is the unique decomposition into
indecomposable summands, up to reordering. Let
\[
\mathcal{F} \cong \bigoplus_{j=1}^s \mathcal{G}_j
\]
be any decomposition into indecomposable coherent sheaves. Since each \(\mathcal{G}_j\) is a direct summand of the \(S_2\)-sheaf
\(\mathcal F\), it is again \(S_2\). In particular, it has no embedded
associated points. Hence 
\(\operatorname{Supp}\mathcal{G}_j\) is a union of irreducible components of
\(X\). Since \(\mathcal G_j\) is indecomposable,
Theorem~\ref{sheaf_indecomposable} implies that
\(\operatorname{Supp}\mathcal G_j\) is connected in codimension \(1\). Hence
\(\operatorname{Supp}\mathcal G_j\subseteq S_i\)
for some \(i\).

For each $i$, set
\[
\mathcal{H}_i := \bigoplus_{\operatorname{Supp}\mathcal{G}_j \subseteq S_i} \mathcal{G}_j.
\]
Then $\operatorname{Supp}\mathcal{H}_i \subseteq S_i$, and for $k \neq i$ we have $\mathcal{H}_k|_{U_i} = 0$. Hence
\[
\mathcal{F}|_{U_i} \cong \mathcal{H}_i|_{U_i} \cong \mathcal{F}^{(i)}|_{U_i}.
\]
Since $S_i \setminus U_i$ has codimension at least $2$ in $S_i$, the extension property (Lemma~\ref{lemma}) implies that
\[\mathcal{H}_i \cong \mathcal{F}^{(i)}.\]
As each $\mathcal{F}^{(i)}$ is indecomposable, it follows that $\mathcal{H}_i$ contains exactly one nonzero summand. Therefore the decomposition is unique up to permutation.
\end{proof}
\begin{remark}
	The equidimensionality assumption ensures that codimension-one points of each
	\(S_i\) are codimension-one points of \(\operatorname{Supp}\mathcal F\).
	Without this assumption, indecomposability of the summands would require a
	modified componentwise local indecomposability condition.
\end{remark}

\begin{theorem}[Structure theorem for graded \(S_2\)-modules]\label{structure_theorem_graded}
	Let \(A=\bigoplus_{\gamma\in\Gamma}A_\gamma\) be a Noetherian
	\(\Gamma\)-graded ring, where \(\Gamma\) is a torsion-free abelian group, and
	let \(M\) be a nonzero finitely generated \(\Gamma\)-graded \(A\)-module
	satisfying \(S_2\). Let
	\[
	\Supp (M)=S_1\cup\cdots\cup S_r
	\]
	be the decomposition of \(\Supp (M)\) into its connected components in
	codimension \(1\). Then there exists a decomposition
	\[
	M\cong \bigoplus_{i=1}^r M^{(i)}, \qquad  \Supp(M^{(i)})=S_i \quad \text{for all } i,
	\]
	into finitely generated \(\Gamma\)-graded \(S_2\)-modules. 
	This decomposition is canonical up to reordering.

	Moreover, assume that \(\Supp(M)\) is equidimensional and that  for every
	homogeneous prime \(	\mathfrak p\in \Supp (M)\) with
	\(	\operatorname{codim}_{\Supp (M)}(\mathfrak p)\le 1\),
	the homogeneous localization $M_{(\mathfrak p)}$ is $\Gamma$-graded indecomposable. Then
	each $M^{(i)}$ is $\Gamma$-graded  indecomposable. In particular, \(M\) admits a decomposition
	into \(\Gamma\)-graded indecomposable summands which is unique up to reordering.
\end{theorem}
\begin{proof}
	Set \(S:=\Supp (M)\), and let
	\[
	S=S_1\cup\cdots\cup S_r
	\]
	be its decomposition into connected components in codimension \(1\). The minimal primes in $S$ are the associated primes of $M$, which are homogeneous by \cite[Thm.~II.7.3]{graded}. Hence each
	\(S_i\), being a union of irreducible components of \(S\), is a homogeneous
	closed subset of \(S\).
	
	Set
	\[
	Y:=\bigcup_{i\ne j}(S_i\cap S_j),
	\qquad
	U:=S\setminus Y,
	\qquad
	U_i:=S_i\setminus Y.
	\]
	Then \(Y\) is homogeneous closed of codimension at least \(2\) in \(S\), and
	\[
	U=U_1\sqcup\cdots\sqcup U_r
	\]
	with each \(U_i\) open and closed in \(U\).
	
	By the \(S_2\)-extension property, 
	\[
	\widetilde M
	\cong
	j_*(\widetilde M|_U)
	\cong
	\bigoplus_{i=1}^r (j_i)_*(\widetilde M|_{U_i}),
	\]
	where \(j:U\hookrightarrow S\) and \(j_i: U_i\hookrightarrow S\) are the inclusions. Taking global sections gives
	\[
	M\cong \bigoplus_{i=1}^r M^{(i)},
	\qquad
	M^{(i)}:=\Gamma(U_i,\widetilde M).
	\]
As in the proof of Theorem~\ref{indecomposable_graded}, the fact that $S\setminus U_i=\bigcup_{j\ne i}S_j$ is   homogeneous closed in $S$ 
  implies that each \(M^{(i)}\) is 
\(\Gamma\)-graded. Since each \(M^{(i)}\) is a direct summand of \(M\), it is finitely generated
and \(S_2\). Moreover, by construction, \(\Supp_A(M^{(i)})=S_i\).
The construction depends only on the components \(S_i\), so the decomposition
is canonical up to reordering.
	
Now assume that \(S\) is equidimensional and that the stated local indecomposability condition holds. Let \(\mathfrak p\in S_i\) be a homogeneous
prime with \(\operatorname{codim}_{S_i}(\mathfrak p)\le 1\). Since \(S\) is
equidimensional, we have
\(\operatorname{codim}_{S}(\mathfrak p)
=
\operatorname{codim}_{S_i}(\mathfrak p)
\le 1 \). Hence 
\((M^{(i)})_{(\mathfrak p)}\cong M_{(\mathfrak p)} \) is \(\Gamma\)-graded indecomposable.
 Since \(\Supp_A(M^{(i)})=S_i\) is connected in
 codimension \(1\), Theorem~\ref{indecomposable_graded} implies that
 \(M^{(i)}\) is \(\Gamma\)-graded indecomposable.

	Finally, the uniqueness of the decomposition into indecomposable summands is
	proved exactly as in Theorem~\ref{structure_S_2 sheaf}, using
	Theorem~\ref{indecomposable_graded} in place of Theorem~\ref{sheaf_indecomposable}.
\end{proof}

Let \(A\) be a local ring admitting a canonical module \(\omega_A\), the ring 
\(B:=\operatorname{End}_A(\omega_A)\)
is the \(S_2\)-ification of \(A\). In the complete local case, \(B\) decomposes as a finite product of complete local rings, inducing a corresponding decomposition of \(\omega_A\); see
\cite[Thm.~3.2(4)]{aoyama} and \cite[(2.2\,k)]{hochster}. Sather-Wagstaff and
Spiroff \cite{counting} used the Hochster--Huneke graph to give a combinatorial
count of these factors: in the complete equidimensional local case, the number
of factors of the \(S_2\)-ification agrees with the number of connected
components in codimension \(1\) of \(\Spec(A/J(A))\).

The following theorem shows that this decomposition phenomenon is not
particular to the complete local case but a consequence of the decomposition of $S_2$ objects. Since both \(\omega_A\) and \(B\) are
finite \(S_2\)-modules with support \(\operatorname{Spec}(A/J(A))\), the
structure theorem gives canonical decompositions of these objects, determined
intrinsically by the codimension-\(1\) connected components of
\(\operatorname{Spec}(A/J(A))\).
\begin{theorem}\label{decomposition_canonical_S_2 ification}
 Let $A$ be a local ring that is a homomorphic image of a Gorenstein local ring. Let $(0) = \bigcap_{p \in \Ass A} Q(p)$ be a primary decomposition of $0$ in $A$ so that $J(A) = \bigcap_{p \in \Assh A} Q(p)$.  Let
 \[
 \Spec(A/J(A))=S_1\cup\cdots\cup S_r
 \]
 be the decomposition into connected components in codimension \(1\).

 For each \(i\), let \(T_i:=S_i\cap \Assh(A)\). Define
 \[
 J_i := \bigcap_{p \in T_i} Q(p), \qquad A_i := A/J_i.
 \]
	\begin{enumerate}
		\item[(a)] The canonical module $\omega_A$ admits a decomposition
	\begin{equation}\label{decomposition_omega_A}
			\omega_A \cong \omega_1 \oplus \cdots \oplus \omega_r,\qquad \Supp_A(\omega_i)=S_i \quad \text{for all } i.
	\end{equation} 
		Each $\omega_i$ is indecomposable, and this gives the unique decomposition of $\omega_A$ into indecomposable $A$-modules, up to permutation.
		
		\item[(b)] Let $B := \operatorname{End}_A(\omega_A)$. Then   
		\[
		B \cong \prod_{i=1}^r B_i, \qquad B_i := \operatorname{End}_A(\omega_i),\qquad \Supp_A(B_i)=S_i \quad \text{for all } i.
		\]
		Each \(B_i\) is connected, equivalently indecomposable as a ring, and is also
		indecomposable as an \(A\)-module. This is the unique decomposition of \(B\) as
		a product of connected rings, or as a direct sum of indecomposable
		\(A\)-modules, up to permutation.
		
	\item[(c)] Each $\omega_i$ is the canonical module of $A_i$,   and $B_i \cong \End_A(\omega_i)$ is the $S_2$-ification of $A_i$. Let $C$ be the cokernel of the natural injection
	\(\alpha \colon A/J(A) \to B,\)
	and let $C_i$ be the cokernel of the natural injection
\(	\alpha_i \colon A_i \to B_i.\)
	Then $\alpha$ factors through $\bigoplus_i \alpha_i$, and we have an exact sequence
	\[
	0 \to A/J(A) \to \bigoplus_{i=1}^r A_i \to C \to \bigoplus_{i=1}^r C_i \to 0.
	\]
		\item[(d)] Each $\omega_i$ is the canonical module of $B_i$.
		
		\item[(e)] If $A$ is complete local, then each $B_i$ is local.
	\end{enumerate}
\end{theorem}
\begin{proof}
	(a) Since \(\omega_A\) is a finitely generated equidimensional \(S_2\)
	\(A\)-module and
	\[
	\operatorname{Supp}_A(\omega_A)=\operatorname{Spec}(A/J(A)),
	\]
	Theorem~\ref{structure_S_2 sheaf} yields the desired decomposition \eqref{decomposition_omega_A}.

Let
\(\overline A:=A/J(A)\). For every prime
\(\mathfrak p\in\operatorname{Spec}\overline A\) with
\(\operatorname{codim}_{\operatorname{Spec}\overline A}(\mathfrak p)\le 1\),
the natural map
\[
\overline A_{\mathfrak p}
\longrightarrow
\operatorname{End}_{A_{\mathfrak p}}\bigl((\omega_A)_{\mathfrak p}\bigr)
\]
is an isomorphism.
Since \(\overline A_{\mathfrak p}\) is
indecomposable as an \(A_{\mathfrak p}\)-module, it follows that $(\omega_A)_{\mathfrak{p}}$ is indecomposable as an $A_{\mathfrak{p}}$-module. 
 Hence the
 second part of Theorem~\ref{structure_S_2 sheaf} applies, so each
 \(\omega_i\) is indecomposable and the decomposition is unique up to
 permutation.
  
  (b)  For \(i\neq j\), we have
  \(  \Ass_A\left(\Hom_A(\omega_i,\omega_j)\right)
  =
  \Supp_A(\omega_i)\cap \Ass_A(\omega_j)=\emptyset,\)
  because \(\Supp_A(\omega_i)=S_i\), while \(\omega_j\) is unmixed thus its
  associated primes are the minimal primes of \(S_j\), none of which lies in
  \(S_i\).
  Hence \[
  \Hom_A(\omega_i,\omega_j)=0\quad \text{for } i\neq j. 
  \] Therefore the decomposition of \(\omega_A\) induces a direct sum
  decomposition of $A$-modules 
  \[
  B\cong \bigoplus_{i=1}^r \End_A(\omega_i).
  \]
  Let \(B_i:=\End_A(\omega_i), \) with the natural ring structures on $B_i$, we also get a product decomposition as rings 
  \[
  B\cong \prod_{i=1}^r B_i.\]
  Moreover, $B_i$ is $S_2$ and $\Supp_A(B_i)=\Supp_A(\omega_i)=S_i$. If \(\mathfrak p\in S_i\) has codimension at most
  \(1\) in \(S_i\), then it has
  codimension at most \(1\) in \(\Spec \overline A\). Hence \(\mathfrak p\) lies
  on no other component \(S_j\), and
  \[
  (B_i)_{\mathfrak p}\cong B_{\mathfrak p}\cong \overline A_{\mathfrak p},
  \]
  which is indecomposable as an \(A_{\mathfrak p}\)-module. Since \(S_i\)
  is connected in codimension \(1\), Theorem~\ref{structure_S_2 sheaf}
  implies that \(B_i\) is indecomposable as an \(A\)-module.

Any nontrivial idempotent of $B_i$   would induce a
nontrivial decomposition of \(\omega_i\) as an \(A\)-module. Since \(\omega_i\)
is indecomposable, \(B_i\) has no nontrivial idempotents; equivalently, \(B_i\)
is connected as a ring. The uniqueness statement follows from the uniqueness
part of Theorem~\ref{structure_S_2 sheaf}.

(c) Let \(\overline A:=A/J(A)\). The diagonal map
\[
i:\overline A\longrightarrow \bigoplus_{i=1}^r A_i
\]
  is  injective, since
\(\bigcap_{i=1}^r J_i=J(A).\) 
Let $N := \operatorname{coker}(i)$.  If
\(\mathfrak p\in\operatorname{Spec}\overline A\) has codimension at most
\(1\), then \(\mathfrak p\) lies in a unique component \(S_i\). Hence \(\overline A_{\mathfrak p}\cong (A_i)_{\mathfrak p},\) \((A_j)_{\mathfrak p}=0  \) for \(j\neq i\), so \(i_{\mathfrak p}\) is an isomorphism.
Thus  \(\dim N\le \dim A-2\).
Since  $\dim \overline{A} = \dim A_i$, it follows   that 
\[
\omega_{\overline{A}} \cong \bigoplus_{i=1}^r \omega_{A_i} .
\]
Moreover, by the decomposition \eqref{decomposition_omega_A}, \(\omega_{\overline A}\cong \omega_A\cong \bigoplus_{i=1}^r \omega_i\). Since  $S_i = \Supp(\omega_i) = \Supp(\omega_{A_i})$, the uniqueness of the decomposition in~(a) implies that $\omega_i \cong \omega_{A_i}$. Therefore, $\operatorname{ann}(\omega_i) = J_i$, and
\(B_i = \End_A(\omega_i) \cong \End_{A_i}(\omega_i)\)
is the $S_2$-ification of $A_i$.

The map $\alpha: \overline{A} \xrightarrow{\alpha} B$ clearly factors through $\bigoplus \alpha_i$, since $\alpha$ is given by multiplication by elements of $\overline{A}$ and hence acts componentwise. That is,
\[
\alpha \colon \overline{A} \xrightarrow{i} \bigoplus_{i=1}^r A_i \xrightarrow{\bigoplus \alpha_i} \bigoplus_{i=1}^r B_i.
\]
Applying the snake lemma to the following commutative diagram with exact rows,
\[
\begin{tikzcd}
	0 \arrow[r] & \overline{A} \arrow[r,"\alpha"] \arrow[d, hook, "i"'] 
	& B \arrow[r] \arrow[d, equal] 
	& C \arrow[r] \arrow[d, two heads, "k"] 
	& 0 \\
	0 \arrow[r] & \bigoplus\limits_{i=1}^r A_i \arrow[r, "\bigoplus \alpha_i"] 
	& \bigoplus\limits_{i=1}^r B_i  \arrow[r] 
	&  \bigoplus\limits_{i=1}^r C_i \arrow[r] 
	& 0
\end{tikzcd}
\]
we obtain $\operatorname{coker} i \cong \ker k  $, and hence the exact sequence
\[
0 \to \overline{A} \to \bigoplus_{i=1}^r A_i \to C \to \bigoplus_{i=1}^r C_i \to 0.
\]

(d) Let \(d=\dim A\). For any finitely generated \(A\)-module \(M\) of
dimension \(d\), we have
\[
K_M\cong \operatorname{Hom}_A(M,\omega_A).
\] 
Since \(\omega_i\) is equidimensional, unmixed and \(S_2\), by \cite[Prop.~3.2]{canonical_module}, the canonical biduality map induces an isomorphism 
\(\omega_i \cong K_{K_{\omega_i}}\).
Moreover,  $\dim B_i = \dim \omega_i =\dim A= d$, thus the canonical module of \(B_i\) is given by
\[
\begin{aligned}
	K_{B_i}
	&\cong \operatorname{Hom}_A(B_i,\omega_A) \\
	&\cong \operatorname{Hom}_A\big(\operatorname{Hom}_A(\omega_i,\omega_i), \omega_A\big) \\
	&\cong \operatorname{Hom}_A\big(\operatorname{Hom}_A(\omega_i,\omega_A), \omega_A\big) \\
	&\cong K_{K_{\omega_i}} \cong \omega_i,
\end{aligned}
\] 
where we used \(\operatorname{Hom}_A(\omega_i,\omega_j)=0\) for \(i\ne j\), as shown in the proof of (b).

(e) 
Since \(A\) is complete local and \(B\) is module-finite over \(A\), the
ring \(B\) is semilocal and complete with respect to its Jacobson radical. Hence 
\[
B \cong B_{\mathfrak n_1}\times\cdots\times B_{\mathfrak n_s},
\]
where $\mathfrak n_1,\dots,\mathfrak n_s$ are the maximal ideals of $B$.
Each factor \(B_{\mathfrak n_j}\) is local, hence connected.
By the uniqueness of product decomposition in (b),
the factors \(B_1,\ldots,B_r\) are precisely \(B_{\mathfrak n_1},\dots,B_{\mathfrak n_s}\), up to
permutation. In
particular, each \(B_i\) is local.
 \end{proof}

\begin{remark}
The assumption that \(A\) is a homomorphic image of a Gorenstein local
ring is used here to ensure the existence of a canonical module and to
place the argument within the usual canonical-module formalism. In a more general scheme-theoretic setting, where a dualizing complex need not exist, an analogous role is played by Kollár's torsion-free \(S_2\)-dualizing
sheaves; see \cite[Thm.~1, Thm.~2]{Kollar2022}.
	
Since the arguments above are governed largely by the \(S_2\)-extension
property and by connectedness in codimension \(1\), one expects the
decomposition theorem here to admit a parallel formulation in terms of
torsion-free \(S_2\)-dualizing sheaves on more general schemes.  
\end{remark}

\begin{corollary}
	Let $C$ be an equidimensional Noetherian $S_2$-ring, and let
	\[
	\Spec C = T_1 \cup \cdots \cup T_s
	\]
	be the decomposition into connected components in codimension $1$.
	Then $C$ admits a canonical product decomposition
	\[
	C \cong \prod_{j=1}^s C_j
	\]
	such that
	\[
	\Supp_C(C_j) = T_j \qquad \text{for all } j.
	\]
	Each $C_j$ is connected, equivalently indecomposable as a ring, and this is the unique decomposition of $C$ as a product of connected rings, up to permutation.
	
	If \(C\) is complete semilocal, then each \(C_j\) is complete local. In
	particular, if \(C\) is complete semilocal and \(\Spec C\) is connected in
	codimension \(1\), then \(C\) is local.
\end{corollary}
\begin{proof}
	Apply Theorem~\ref{structure_S_2 sheaf} to \(C\) as a finitely generated \(S_2\)-module
	over itself. This gives a canonical decomposition
	\(	C\cong \bigoplus_{j=1}^s C_j\)
	as \(C\)-modules with \(\Supp_C(C_j)=T_j\). Each summand is of the form
	\(C_j=Ce_j\), where \(e_j\in C\) is the corresponding idempotent, hence
	\(C_j\) is naturally a ring with identity \(e_j\) and the above direct sum
	is also a product decomposition of rings. The connectedness and uniqueness
	statements follow from the uniqueness part of
	Theorem~\ref{structure_S_2 sheaf}.
	
	If \(C\) is complete semilocal, then it decomposes as a product of its
localizations at maximal ideals 
\(C\cong \prod_{\mathfrak n\in \operatorname{mSpec} C} C_{\mathfrak n}.\)
Since both the factors \(C_j\) and the local factors \(C_{\mathfrak n}\) are
connected, the uniqueness of the product decomposition into connected rings
implies that each \(C_j\) is one of the \(C_{\mathfrak n}\). Hence each
\(C_j\) is complete local.
\end{proof}

\subsection{A discussion on connectedness theorems}\label{connectedness_section}
Faltings' connectedness theorem \cite[Thm.~6]{falting}  in local form
says that if \((A,\mathfrak m)\) is a complete local domain and \(\mathfrak a\)
is generated by at most \(\dim A-2\) elements, then the punctured spectrum
\(\Spec(A/\mathfrak a)\setminus\{\mathfrak m/\mathfrak a\}\)
is connected.
Hochster and Huneke generalized this result by replacing the
domain condition with the indecomposability of the top local cohomology
module. More precisely, if \((A,\mathfrak m)\) is a complete equidimensional
local ring of dimension \(n\), and if \(H^n_{\mathfrak m}(A)\), equivalently
the canonical module \(\omega_A\), is indecomposable, then for every proper
ideal \(\mathfrak a\) generated by at most \(n-2\) elements, the punctured
spectrum 
\(\Spec(A/\mathfrak a)\setminus\{\mathfrak m/\mathfrak a\}\)
is connected
\cite[Thm.~3.3]{hochster}.

Hochster and Huneke also studied equivalent conditions for the canonical
module to be indecomposable, which we recall below. In particular, when
\((A,\mathfrak m)\) is complete and equidimensional of dimension \(n\), the
indecomposability of \(H^n_{\mathfrak m}(A)\) is equivalent to connectedness in
codimension \(1\) of \(\Spec A\). Once this codimension-\(1\) connectedness is
known, the usual punctured-spectrum connectedness conclusions follow from
Grothendieck's connectedness theorem
\cite[Thm.~19.2.10]{local cohomology}. Indeed, if
\(\operatorname{ara}(\mathfrak a)\le n-2\), then
\[
c(A/\mathfrak a)
\ge
\min\{c(A),\operatorname{sdim}A-1\}
-\operatorname{ara}(\mathfrak a)
\ge 1,
\]
and hence
\(\Spec(A/\mathfrak a)\setminus\{\mathfrak m/\mathfrak a\}\)
is connected. The condition \(\operatorname{ara}(\mathfrak a)\le n-2\) may also be replaced
by the cohomological dimension condition
\(\operatorname{cd}(\mathfrak a,A)\le n-2\), as in the discussion preceding
\cite[Thm.~19.2.10]{local cohomology}; see also \cite[Thm.~1.6]{varbaro}.

\begin{theorem}\cite[Theorem~3.2]{hochster}\label{hochster}
	Let $(A,\mathfrak{m})$ be a complete local equidimensional ring of dimension $n$. The following are equivalent:
	\begin{enumerate}
		\item[(a)] $H^n_{\mathfrak{m}}(A)$ is indecomposable.
		
		\item[(b)] The canonical module $\omega_A$ of $A$ is indecomposable.
		
		\item[(c)] The $S_2$-ification $S \cong \operatorname{Hom}_A(\omega_A,\omega_A)$   is local.
		
		\item[(d)] For every ideal $J$ of height at least $2$, $\operatorname{Spec} A \setminus V(J)$ is connected.
		
		\item[(e)] For any two primes $\mathfrak{p}, \mathfrak{q} \in \operatorname{Assh} A$, there is a chain
		\[
		\mathfrak{p}=\mathfrak{p}_0,\mathfrak{p}_1,\dots,\mathfrak{p}_s=\mathfrak{q}
		\]
		in $\operatorname{Assh} A$ such that
		\(	\operatorname{ht}(\mathfrak{p}_{i-1}+\mathfrak{p}_i)=1\)
		for all $i $.
	\end{enumerate}
\end{theorem}

Schenzel later reproved this theorem from the perspective of ideal transforms
\cite[Cor.~6.3]{connectedness}. If \(I\) defines the non-\(S_2\) locus, then
the \(S_2\)-ification satisfies
\(\operatorname{Hom}_A(\omega_A,\omega_A)\cong D_I(A),\)
so the connectedness condition can be interpreted through the
indecomposability of an ideal transform; see Remark \ref{remark:section_functor_connectedness}.

For a finitely generated \(A\)-module \(M\), the analogue of the
\(S_2\)-ification of a ring is the \(S_2\)-hull of \(M\)
\cite[Def.~9.3]{Kollar2023}. When \(A\) admits a canonical module, this
hull is given by \(K_{K_M}\), as defined in 
\eqref{def_canonical_module}; see \cite[Section~3]{canonical_module}.
The theorem below gives a module-theoretic form of the Hochster--Huneke
equivalences, as a consequence of the structure of finite
\(S_2\)-modules. In particular, it
recovers the Hochster--Huneke equivalences in their setting. In contrast with the Hochster--Huneke
theorem, where the ring structure and localness of the \(S_2\)-ification play
a central role, the following shows that the
codimension-one connectedness criterion is already controlled by the
indecomposability of the underlying finite \(S_2\)-module, rather than by a
local ring structure on the \(S_2\)-ification.
\begin{theorem}\label{canonical_M}
	Let \((A,\mathfrak m)\) be a local ring that is a homomorphic image of a
	Gorenstein local ring, and let \(M\) be a finitely generated \(A\)-module.
	Suppose that \(\overline M_{\mathfrak p}\) is indecomposable for every
	\(\mathfrak p\in \Supp\overline M\) with
	\(	\operatorname{codim}_{\Supp\overline M}(\mathfrak p)\le 1.\)
	Then the following are equivalent:
	\begin{enumerate}
		\item[\rm(a)] \(K_M\) is indecomposable;
		\item[\rm(b)] \(K_{K_M}\) is indecomposable;
		\item[\rm(c)] \(\Supp\overline M\) is connected in codimension \(1\).
	\end{enumerate}
\end{theorem}
\begin{proof}
	By Proposition~\ref{prop_canonical_module}(b), the module \(K_M\) is
	equidimensional, unmixed and   \(S_2\). Hence the canonical
	biduality map \cite[Prop.~3.2]{canonical_module}, applied to \(K_M\),
	gives an isomorphism
	\[
	K_M\cong K_{K_{K_M}}.
	\]
	Let \(R\) be a Gorenstein local ring of dimension \(n\) mapping onto \(A\), and set \(t=\dim M\). Recall that, for a finitely generated \(A\)-module
	\(N\) of dimension \(t\), its canonical module is
	\(	K_N:=\operatorname{Ext}_R^{n-t}(N,R).\) Since \(K_M\) is unmixed and equidimensional of dimension \(t\), every nonzero direct summand of \(K_M\) also has dimension \(t\). Thus applying 
	\(	\Ext_R^{n-t}(-,R)\)
	to a nontrivial decomposition of \(K_M\) gives a nontrivial decomposition of \(K_{K_M}\).
	
	Conversely, \(K_{K_M}\) is also unmixed and equidimensional of dimension \(t\). Hence applying \(\Ext_R^{n-t}(-,R)\) to a nontrivial decomposition of \(K_{K_M}\) together with the biduality isomorphism
	\(	K_M\cong K_{K_{K_M}},\)
	gives a nontrivial decomposition of \(K_M\). 
	
	Therefore \(K_M\) is indecomposable if and only if \(K_{K_M}\) is indecomposable.
	
	By Proposition~\ref{prop_canonical_module}(b),  
\(	\operatorname{Supp}K_{K_M}=\operatorname{Supp}\overline M.\)
	  Since \(\dim J(M)<t\), we have
\(	K_M\cong K_{\overline M}.\)
	Moreover, by \cite[Prop.~3.3]{canonical_module}, the natural biduality
	map
	\[
	\overline M\longrightarrow K_{K_{\overline M}}\cong K_{K_M}
	\]
	is injective with cokernel of dimension at most \(t-2\). Hence,
	for every prime \(\mathfrak p\in \operatorname{Supp}\overline M\) with 
	\(	\operatorname{codim}_{\operatorname{Supp}\overline M}(\mathfrak p)\le 1,\)
	localizing the biduality map gives an isomorphism
	\(	\overline M_{\mathfrak p}
	\cong
	(K_{K_M})_{\mathfrak p}.\)
	By hypothesis, these localizations are indecomposable.
	
	Since \(K_{K_M}\) is \(S_2\) and has support
	\(\operatorname{Supp}\overline M\), Theorem~\ref{sheaf_indecomposable}
	applies and gives
	\[
	K_{K_M}\text{ is indecomposable}
	\Longleftrightarrow
	\operatorname{Supp}\overline M
	\text{ is connected in codimension }1.
	\]
\end{proof}

In the special case \(M=A\), this gives the Hochster--Huneke equivalences in
Theorem~\ref{hochster}. More precisely, we obtain the equivalence between
indecomposability of \(\omega_A\), indecomposability of the \(S_2\)-ification
as an \(A\)-module, and connectedness in codimension \(1\) of
\(\operatorname{Spec}(A/J(A))\), without assuming completeness. In the complete
case this also recovers the localness of the \(S_2\)-ification and the
indecomposability of \(H^d_{\mathfrak m}(A)\).

\begin{corollary}\label{cor_hochster}
	Let \((A,\mathfrak m)\) be a local ring which is a homomorphic image of
	a Gorenstein local ring, and let \(\omega_A\) be the canonical module of
	\(A\). Then the following are equivalent:
	\begin{enumerate}
		\item[(a)] \(\omega_A\) is indecomposable;
		\item[(b)] \(\omega_{\omega_A}\cong  \operatorname{End}_A(\omega_A)\)
		is indecomposable;
		\item[(c)] \(\operatorname{Spec}(A/J(A))\) is connected in codimension \(1\).
	\end{enumerate}
	If \(A\) is complete, then these conditions are also equivalent to
\(\operatorname{End}_A(\omega_A)\) being a local ring, and to
	\(H^{\dim A}_{\mathfrak m}(A)\) being indecomposable.
\end{corollary}

\begin{proof}
	Apply Theorem~\ref{canonical_M} to \(M=A\). Then
	\[
	\overline M=A/J(A), \qquad K_{K_M} = \omega_{\omega_A}.\]
	Moreover, since \(\dim \omega_A=\dim A\), we have
	\[
	\omega_{\omega_A}
	\cong
\operatorname{End}_A(\omega_A).
	\]
	The local indecomposability hypothesis in Theorem~\ref{canonical_M} is
	automatic,  \((A/J(A))_{\mathfrak p}\) is indecomposable as an $A_{\mathfrak p}$ module for any 
	\(\mathfrak p\in \operatorname{Spec}(A/J(A))\). 
  Hence
	Theorem~\ref{canonical_M} gives the equivalence of \((a)\), \((b)\), and
	\((c)\).
	
Now assume that \(A\) is complete local, and set
\(B:=\operatorname{End}_A(\omega_A).\)
By Theorem~\ref{decomposition_canonical_S_2 ification}\textup{(b),(e)}, \(B\)
is local if and only if \(\operatorname{Spec}(A/J(A))\) is connected in
codimension \(1\). Thus localness of \(B\) is equivalent to the above
conditions.

Finally, by local duality,
\[
\omega_A^\vee\cong H^d_{\mathfrak m}(A),
\qquad
H^d_{\mathfrak m}(A)^\vee\cong \omega_A
\]
where \(d=\dim A\). Hence \(H^d_{\mathfrak m}(A)\) is indecomposable if and only
if \(\omega_A\) is indecomposable.
\end{proof}

We record the following Hochster--Huneke type consequence, which follows
from connectedness in codimension \(1\) together with Grothendieck's connectedness
theorem. It recovers the usual complete case, for instance
\cite[Corollary~1.7]{varbaro}, and gives a non-complete formulation.
\begin{corollary}\label{universally catenary}
	Let \((A,\mathfrak m)\) be a universally catenary local ring, and let \(d=\dim A\). Suppose that
	\(H^d_{\mathfrak m}(A)\) is indecomposable. Then for any proper  ideal
	\(\mathfrak a\subseteq A\) with
	\(	\operatorname{cd}(\mathfrak a,A)\le d-2\), 
	\( 	\operatorname{Spec}(A/(J(A)+\mathfrak a))\setminus\{\mathfrak m\}\)
	is connected. In particular, if \(A\) is equidimensional, then the punctured spectrum 
	\(	\operatorname{Spec}(A/\mathfrak a)\setminus\{\mathfrak m\}\)
	is connected.
\end{corollary}
\begin{proof}
Since \(\dim J(A)<d\), we have
\(H^d_{\mathfrak m}(A)\cong H^d_{\mathfrak m}(A/J(A)).\)
	Thus the indecomposability assumption is preserved after replacing
	\(A\) by \(A/J(A)\). Moreover, 
	\(	\operatorname{cd}(\mathfrak a,A/J(A))\le \operatorname{cd}(\mathfrak a,A).\) Therefore we may assume that \(A\) is equidimensional of
	dimension \(d\).
	
	Since \(A\) is  
	universally catenary,  it is   formally catenary by \cite[Thm.~31.7]{matsumura}.  Therefore, for any minimal prime 	\(\mathfrak p\) of \(A\), \(\widehat A/\mathfrak p \widehat A\) is equidimensional of dimension $d$.
	
	Now let \(Q\) be a minimal prime of \(\widehat A\). By faithful flatness
	and going-down, \(Q\cap A=\mathfrak p\) is a minimal prime of \(A\), and
	\(Q\) is minimal over \(\mathfrak p\widehat A\). Since
	\(\widehat A/\mathfrak p\widehat A\) is equidimensional of dimension
	\(d\), we have
	\(	\dim \widehat A/Q=d.\)
    Thus \(\widehat A\) is equidimensional of dimension \(d\).

Moreover,  $	H^d_{\widehat{\mathfrak m}}(\widehat A)\cong H^d_{\mathfrak m}(A)$. Since \(H^d_{\mathfrak m}(A)\) is Artinian, decomposability is the same
over \(A\) and over \(\widehat A\). Hence
\(H^d_{\widehat{\mathfrak m}}(\widehat A)\) is indecomposable.
  By
	Theorem~\ref{hochster} or \ref{canonical_M}, \(\Spec \widehat A\) is connected in codimension \(1\);
	equivalently,
	\(	c(\widehat A)\ge d-1.\)
	Also \(\operatorname{sdim}\widehat A=d\). By Grothendieck's connectedness theorem in the form
	\cite[Thm.~1.6]{varbaro}, together with
	\cite[Lem.~19.3.1]{local cohomology}, we have
	\[
		c(  A/\mathfrak a )\ge 
 	c(\widehat A/\mathfrak a\widehat A)
	\ge
	\min\{c(\widehat A),\operatorname{sdim}\widehat A-1\}
	-
	\operatorname{cd}(\mathfrak a\widehat A,\widehat A).
	\]
	Since
	\(	\operatorname{cd}(\mathfrak a\widehat A,\widehat A)
	=
	\operatorname{cd}(\mathfrak a,A)
	\le d-2,\)
	we get
\(	c(  A/\mathfrak a )\ge d-1-(d-2)=1.\)
Thus \(\Spec(A/\mathfrak a)\setminus\{\mathfrak m\}\) is connected.
\end{proof}

\begin{remark}\label{remark:section_functor_connectedness}
	One mechanism underlying the connectedness and structure theorems for
	\(S_2\)-sheaves is the extension property. For an \(S_2\)-sheaf this
	extension is rigid across subsets of codimension at least \(2\), but even for
	a general coherent sheaf the extension \(j_*(\mathcal F|_U)\) can detect
	connectedness of \(U\).

	Let \(X\) be a Noetherian scheme, let \(\mathcal F\) be a coherent sheaf on
	\(X\), and let \(U\subseteq \Supp\mathcal F\) be an open subset. Write
	\(	U=U_1\sqcup\cdots\sqcup U_r\)
	for the decomposition into connected components. Let
	\(j:U\hookrightarrow X\)  and \(j_i: U_i\hookrightarrow X\)   be the inclusions.
	Then
	\(	j_*(\mathcal F|_U)
	\cong
	\bigoplus_{i=1}^r (j_i)_{ *}(\mathcal F|_{U_i}).\)
	In particular, if \(j_*(\mathcal F|_U)\) is indecomposable, then \(U\) is
	connected.
	
Conversely, the same localization-and-support argument as in
Theorem~\ref{sheaf_indecomposable} gives the following criteria.
	\begin{enumerate}
		\item[(a)] If \(\mathcal F_x\) is indecomposable as an
		\(\mathcal O_{X,x}\)-module for every \(x\in U\), then connectedness of
		\(U\) implies that \(j_*(\mathcal F|_U)\) is indecomposable.
		
		\item[(b)] Suppose that every associated point of \(\mathcal F|_U\) has
		codimension at most \(k\) in \(U\). If \(\mathcal F_x\) is indecomposable
		for every \(x\in U\) with \(\operatorname{codim}_U(x)\le k\), then
		connectedness of \(U\) in codimension \(k\) implies that
		\(j_*(\mathcal F|_U)\) is indecomposable.
	\end{enumerate}
	
For \(X=\Spec A\), \(U=X\setminus V(I)\), and \(\mathcal F=\widetilde M\), this
recovers the ideal-transform viewpoint  in \cite{connectedness}
\[
D_I(M)\cong \Gamma(U,\widetilde M)
\cong \Gamma(X,j_*(\widetilde M|_U)).
\]

	The preceding sheaf-theoretic criteria also have a formal analogue.  
Namely, if \(\mathfrak X\) is a Noetherian formal scheme, \(\mathfrak F\) is a
coherent \(\mathcal O_{\mathfrak X}\)-module, and
\(U\subseteq \Supp \mathfrak F\) is open, then analogous statements apply to
the extension \(j_*(\mathfrak F|_U)\), where \(j:U\hookrightarrow\mathfrak X\)
is the inclusion.
In the affine case one  may take  \(\mathfrak X\) to be the formal completion of \(\Spec A\) along \(V(I)\), where 
\((A,\mathfrak m)\) is a local ring and \(I\subseteq A\) is an ideal.  
 For a finite \(A\)-module \(M\), let
\( {\mathcal M}^\Delta\) denote the completion of the coherent sheaf $\widetilde{M}$ on \(\mathfrak X\).  If
\(U:=\Spec A\setminus\{\mathfrak m\},
W:=\Supp_A(M/IM)\cap U,\) then the formal transform
\[
D^I(M)
:=
\varprojlim_\alpha D_{\mathfrak m}(M/I^\alpha M)
\cong
\varprojlim_\alpha
\Gamma(U,\widetilde{M/I^\alpha M}) \cong
\varprojlim_\alpha
\Gamma\bigl(W,(\widetilde{M/I^\alpha M})|_W\bigr) 
\]
may be viewed as the module of sections of \(\mathcal M^\Delta\) on the open
subset \(U\cap V(I)\) of the formal scheme \(\mathfrak X\).  
 If
\(W=W_1\sqcup\cdots\sqcup W_r\)
is the decomposition into connected components, then\[
D^I(M)
\cong
\bigoplus_{i=1}^r
\Gamma(W_i,\mathcal M^\Delta|_{W })
\cong
\bigoplus_{i=1}^r
\varprojlim_\alpha
\Gamma\bigl(W_i,(\widetilde{M/I^\alpha M})|_{W }\bigr).
\]  If each component contributes a nonzero section module, then
indecomposability of \(D^I(M)\) forces \(W\) to be connected. Unlike the case of
sections of a quasi-coherent sheaf on an affine scheme, this nonvanishing
condition is not guaranteed, since \(D^I(M)\) is an inverse limit of section
modules. For \(M=A\), the condition is automatic:
each component carries the compatible identity section. In particular, when
\(M=A\) and \(A\) is complete local, Schenzel proves that
indecomposability of \(D^I(A)\) as a ring is equivalent to the connectedness
of \(V(I)\setminus\{\mathfrak m\}\); see
\cite[Thm.~8.2]{connectedness}.

Thus the ideal-transform and formal-transform connectedness criteria are
instances of the same sheaf-extension principle, although passing to modules of
sections can behave differently in the two settings. The special feature of
\(S_2\)-sheaves is that the extension property identifies the original finite
object itself with such an extension across all complements of  subsets of codimension at least \(2\).
\end{remark}

\section{Deficiency Modules and the Non-$S_2$ Locus}
In this section, we  continue with the set-up of Section \ref{set-up_non-S_2}. Thus
\(0\to A\to B\to C\to 0\)
is the \(S_2\)-ification sequence. Assume \(C\neq 0\), and let \(t=\dim C\).

We study the relationship between the deficiency module \(K^{t+1}(A)\) and
the canonical module \(K^t(C)\) of the \(S_2\)-defect module \(C\). We show
that \(K^t(C)\) is the \(S_2\)-hull of \(K^{t+1}(A)\), and discuss when the
natural map between them is an isomorphism, both globally and locally. We
conclude with examples.

The first step is to compare the top-dimensional support of \(K^{t+1}(A)\)
with that of \(C\).

\begin{proposition}\label{prop:assh-equality}
	Let \((A,\mathfrak m)\) be a local ring that is a homomorphic image of
	a Gorenstein local ring \(R\), and assume \(J(A)=0\). Let \(C\) be as in
	\eqref{exact_ABC}. Assume \(C\neq 0\), and let \(t=\dim C\).  Then
	\[
	\Assh_A K^{t+1}(A)=\Assh_A C.
	\]
\end{proposition}

\begin{proof}
Applying $\Ext_R^{n-\bullet}(-,R)$ to the short exact sequence \eqref{exact_ABC}, we obtain an exact sequence
\[
K^{t+1}(B)\longrightarrow K^{t+1}(A)\xrightarrow{\varphi} K^t(C)\longrightarrow K^t(B).
\]

 For any prime
\(\mathfrak q\in \Spec A\), we have
\begin{equation}\label{s_2_dim}
	\bigl(K^{j}(B)\bigr)_{\mathfrak q}
	\;\cong\;
	\Ext^{\,n-j}_{R_{\mathfrak q}}\!\left(B_{\mathfrak q}, R_{\mathfrak q}\right)
	\;\cong\;
	K^{\dim R_{\mathfrak q}-(n-j)}_{A_{\mathfrak q}}\!\left(B_{\mathfrak q}\right)
	\;=\;
	K^{\,j-\dim A/\mathfrak q}_{A_{\mathfrak q}}\!\left(B_{\mathfrak q}\right),
\end{equation}
where we use
\(\dim A/\mathfrak q
=
\dim R/\mathfrak q
=
n-\dim R_{\mathfrak q} \)  since $R$ is Cohen--Macaulay.

Since \(B\) is \(S_2\), local duality gives
\[
K^i_{A_{\mathfrak q}}(B_{\mathfrak q})=0
\qquad\text{for } i<\min\{2,\dim B_{\mathfrak q}\}.
\]
Also \(\Supp_A B=\Spec A\), so
\(\dim B_{\mathfrak q}=\dim A_{\mathfrak q}\).
Since \(A\) is a quotient of a Gorenstein local ring, it is universally
catenary. Moreover,   \(J(A)=0\) implies that \(A\) is
 equidimensional, thus
\[
\dim B_{\mathfrak q}
=
\dim A_{\mathfrak q}
=
\dim A-\dim A/\mathfrak q.
\]
Let \(j=t,t+1\), and suppose \(\dim A/\mathfrak q\ge t\), then
\(j-\dim A/\mathfrak q<2. \)
Also, since \(t\le \dim A-2\), we have
\[
j-\dim A/\mathfrak q
\le
t+1-\dim A/\mathfrak q
\le
\dim A-\dim A/\mathfrak q-1
=
\dim B_{\mathfrak q}-1.
\]
Thus the right-hand side of \eqref{s_2_dim} vanishes for \(j=t,t+1\).
Hence, for every \(\mathfrak q\) with \(\dim A/\mathfrak q\ge t\),
\begin{equation}\label{localization_K_t}
	\bigl(K^{t+1}(B)\bigr)_{\mathfrak q}
	=
	\bigl(K^t(B)\bigr)_{\mathfrak q}
	=0.
\end{equation}
Therefore \(\varphi_{\mathfrak q}\) is an isomorphism
\(\bigl(K^{t+1}(A)\bigr)_{\mathfrak q}
\cong
\bigl(K^t(C)\bigr)_{\mathfrak q}\)
for all such \(\mathfrak q\).
 
Since \(K^t(C)\) is the canonical module of \(C\), by Proposition \ref{prop_canonical_module},  
\[
\dim K^t(C)=\dim C=t
\qquad\text{and}\qquad
\Ass_A K^t(C)=\Assh_A C.
\]
The preceding   isomorphism shows that \(K^{t+1}(A)\) and
\(K^t(C)\) have the same support in dimensions at least \(t\). This implies that 
\(K^{t+1}(A)\) has dimension \(t\), and its top dimensional support agrees
with that of \(K^t(C)\). Therefore
\[
\Assh_A K^{t+1}(A)
=
\Assh_A K^t(C)
=
\Assh_A C.
\]
\end{proof}
\begin{remark}
	The vanishing \eqref{localization_K_t} also follows immediately from
	Proposition~\ref{prop_canonical_module}(c), applied to the \(S_2\)-module
	\(B\):
	\[
	\dim K^{t+1}(B)\le t-1
	\quad\text{and}\quad
	\dim K^t(B)\le t-2.
	\]
 We included the direct proof
above for completeness.
\end{remark}

\begin{theorem}\label{thm_S_2hull}
Let \((A,\mathfrak m)\) be a local ring that is a homomorphic image of a
Gorenstein local ring, and assume that \(J(A)=0\). Let \(B\) and \(C\) be as in
\eqref{exact_ABC}. Assume \(C\neq 0\), and let \(t=\dim C\). 
Then the connecting morphism
\[
\varphi\colon K^{t+1}(A)\longrightarrow K^t(C)
\]
arising from \eqref{exact_ABC} is the \(S_2\)-hull of \(K^{t+1}(A)\) in the
sense of \cite[Def.\ 9.3]{Kollar2023}. In particular,
\[
K^t(C)\cong K_{K_{K^{t+1}(A)}}.
\]
Moreover, for every ideal \(I\) such that
\(\operatorname{ann}_A K^t(B)
\subseteq I
\subseteq
\operatorname{ann}_A(\operatorname{coker}\varphi),\)
there is an isomorphism
\[
K^t(C)\cong D_I\bigl(\overline{K^{t+1}(A)}\bigr).
\]
\end{theorem}
\begin{proof}
Applying \(\Ext_R^{n-\bullet}(-,R)\) to \eqref{exact_ABC}, we obtain an exact
sequence
\[
K^{t+1}(B)\longrightarrow K^{t+1}(A)
\xrightarrow{\varphi} K^t(C)\longrightarrow K^t(B).
\]
Since \(B\) is \(S_2\), Proposition~\ref{prop_canonical_module}(c) gives
\[
\dim K^{t+1}(B)\le t-1
\qquad\text{and}\qquad
\dim K^t(B)\le t-2.
\]
Hence
\[
\dim \ker\varphi\le t-1
\qquad\text{and}\qquad
\dim \operatorname{coker}\varphi\le t-2.
\]
On the other hand, \(K^t(C)\) is the canonical module of \(C\). Thus, by
Proposition~\ref{prop_canonical_module}(b), \(K^t(C)\) is \(S_2\),
equidimensional, and unmixed. Moreover, by Proposition~\ref{prop:assh-equality},
\[
\dim K^{t+1}(A)=\dim K^t(C)=\dim C=t.
\]
Therefore \(\varphi\colon K^{t+1}(A)\to K^t(C)\) satisfies the defining
properties of the \(S_2\)-hull of \(K^{t+1}(A)\).

On the other hand, by \cite[Prop.~3.2, 3.3]{canonical_module}, the canonical biduality map
\(\tau_{K^{t+1}(A)}\colon K^{t+1}(A)\to K_{K_{K^{t+1}(A)}}\)
  also satisfies the defining
  properties of the  \(S_2\)-hull of \(K^{t+1}(A)\). By uniqueness of the \(S_2\)-hull, there is an isomorphism
\[
K^t(C)\cong K_{K_{K^{t+1}(A)}}.
\]

Before proving the ideal-transform description, we note the following. Since $K^t(C)$  has no nonzero submodule of dimension $<t$, every submodule of \(K^{t+1}(A)\) of dimension \(<t\) is contained in \(\ker\varphi\). It follows that \(\ker\varphi\) is the maximal
submodule of \(K^{t+1}(A)\) of dimension \(<t\). Hence \(K^{t+1}(A)/\ker\varphi \cong \overline{K^{t+1}(A)}\), and \(\varphi\) induces an exact sequence 
\[ 0\to  \overline{K^{t+1}(A)} \xrightarrow{  \varphi } K^t(C)\to Q\to 0, \] 
where $Q:=\operatorname{coker}\varphi$ and $\dim Q\le t-2$.

Set
\[
M:=\overline{K^{t+1}(A)},
\qquad
J:=\operatorname{ann}_A Q,
\qquad
U:=\Spec A\setminus V(J),
\]
and let \(j:U\hookrightarrow \Spec A\) be the inclusion. Since \(\Supp Q=V(J)\), we have \(\widetilde Q|_U=0\). Therefore
\[
\widetilde{M}|_U\cong \widetilde{K^t(C)}|_U.
\]

The \(S_2\)-extension property as in Lemma~\ref{lemma} gives  
\[
\widetilde{K^t(C)}
\cong
j_*\bigl(\widetilde{K^t(C)}|_U\bigr)
\cong
j_*\bigl(\widetilde{M}|_U\bigr).
\]
Taking global sections yields
\[
K^t(C)\cong \Gamma(U,\widetilde{M})\cong D_J(M).
\]

Now let \(I\) be an ideal such that
\[
\operatorname{ann}_A K^t(B)\subseteq I\subseteq J.
\]
Then \(V(J)\subseteq V(I)\) and 
\(\dim A/I\le \dim K^t(B)\le t-2\).
Applying the same \(S_2\)-extension argument to
\(\Spec A\setminus V(I)\), we obtain
\[
K^t(C)\cong D_I(M)
=
D_I\bigl(\overline{K^{t+1}(A)}\bigr).
\]
\end{proof}
Before turning to examples, we record two interpretations of the comparison map
\(\varphi\).
\begin{remark}[Spectral sequence interpretation of \(\varphi\)]
	Let \(I=\operatorname{ann}_A C\). Since \(\operatorname{ht} I\ge 2\) and
	\(J(A)=0\), we have \(H_I^0(A)=0\), and the exact sequence
	\[
	0\to H_I^0(A)\to A\to D_I(A)\cong B\to H_I^1(A)\to 0
	\]
	gives \(C\cong H_I^1(A)\). Hence the Grothendieck spectral sequence \cite[5.8.3]{weibel}
	\[
	E_2^{p,q}=H_{\mathfrak m}^p(H_I^q(A))
	\Longrightarrow H_{\mathfrak m}^{p+q}(A)
	\]
	has an edge morphism
	\[
	H_{\mathfrak m}^t(C)\longrightarrow H_{\mathfrak m}^{t+1}(A),
	\]
	whose Matlis dual is the map
	\[
	\varphi\colon K^{t+1}(A)\to K^t(C).
	\]
\end{remark}
 
 The second interpretation is local and relates the comparison maps to the local depth
 of the \(S_2\)-ification.
\begin{proposition}\label{prop:stratified-comparison}
	Let \((A,\mathfrak m)\) be a local ring that is a homomorphic image of a
	Gorenstein local ring. Assume that \(J(A)=0\). Assume \(C\neq 0\), and let \(t=\dim C\). Let
	\(\mathfrak p\in \Supp^{\mathrm{top}}(C)\), and suppose that
	\[
	\dim A/\mathfrak p=t-s
	\]
	for some integer \(s\ge 0\). Then the following are equivalent:
	\begin{enumerate}
		\item[(a)]  \(\depth B_{\mathfrak p}\ge s+2\);
		\item[(b)] the connecting maps induce isomorphisms
\[ 	H^i_{\mathfrak pA_{\mathfrak p}}(C_{\mathfrak p})
\cong
H^{i+1}_{\mathfrak pA_{\mathfrak p}}(A_{\mathfrak p})
\qquad\text{for all }0\le i\le s;\]
		\item[(c)] the localized connecting maps
		\[
		\bigl(K^{t-s+i+1}(A)\bigr)_{\mathfrak p}
		\longrightarrow
		\bigl(K^{t-s+i}(C)\bigr)_{\mathfrak p}
		\]
		are isomorphisms for all \(0\le i\le s\).
	\end{enumerate}
\end{proposition}
\begin{proof}
	Since \(\mathfrak p\in \Supp^{\mathrm{top}}(C)\) and
	\(\dim A/\mathfrak p=t-s\), we have \(\dim C_{\mathfrak p}=s\). Localizing
	\eqref{exact_ABC} at \(\mathfrak p\), we get
	\[
	0\to A_{\mathfrak p}\to B_{\mathfrak p}\to C_{\mathfrak p}\to 0.
	\]
	The associated long exact sequence of local cohomology gives an exact sequence 
	\[
	\cdots\to
	H^i_{\mathfrak pA_{\mathfrak p}}(B_{\mathfrak p})
	\to
	H^i_{\mathfrak pA_{\mathfrak p}}(C_{\mathfrak p})
	\to
	H^{i+1}_{\mathfrak pA_{\mathfrak p}}(A_{\mathfrak p})
	\to
	H^{i+1}_{\mathfrak pA_{\mathfrak p}}(B_{\mathfrak p})
	\to\cdots .
	\]
	Thus \(\depth B_{\mathfrak p}\ge s+2\)  implies the stated isomorphisms  for
	\(0\le i\le s\) in (b).
	
	Conversely, assume that these connecting maps in (b) are isomorphisms for
	\(0\le i\le s\). Since \(B_{\mathfrak p}\) is \(S_2\) and
	\[
	\dim B_{\mathfrak p}=\dim A_{\mathfrak p}=d-(t-s)\ge s+2,
	\]
	we have
	\begin{equation}\label{B_p_S_2}
			H^0_{\mathfrak pA_{\mathfrak p}}(B_{\mathfrak p})
		=
		H^1_{\mathfrak pA_{\mathfrak p}}(B_{\mathfrak p})=0.
	\end{equation}
	For \(2\le j\le s\), the vanishing
	\(	H^j_{\mathfrak pA_{\mathfrak p}}(B_{\mathfrak p})=0\)
	follows from the isomorphisms for \(i=j-1\) and \(i=j\), together with the
	long exact sequence.

	Finally, since \(\dim C_{\mathfrak p}=s\), we have \(H^{s+1}_{\mathfrak pA_{\mathfrak p}}(C_{\mathfrak p})=0\).  The isomorphism for \(i=s\) then gives 
	\(	H^{s+1}_{\mathfrak pA_{\mathfrak p}}(B_{\mathfrak p})=0.\)
	Hence \(\depth B_{\mathfrak p}\ge s+2\).
	
	The equivalence of (b) with the deficiency-module formulation (c) follows by local
	duality. Indeed,
	\[
	H^i_{\mathfrak pA_{\mathfrak p}}(C_{\mathfrak p})^\vee
	\cong
	K^i_{A_{\mathfrak p}}(C_{\mathfrak p}),
	\qquad
	H^{i+1}_{\mathfrak pA_{\mathfrak p}}(A_{\mathfrak p})^\vee
	\cong
	K^{i+1}_{A_{\mathfrak p}}(A_{\mathfrak p}),
	\]
	and, since \(\dim A/\mathfrak p=t-s\),
	\[
	\bigl(K^{t-s+i}(C)\bigr)_{\mathfrak p}
	\cong
	K^i_{A_{\mathfrak p}}(C_{\mathfrak p}),
	\qquad
	\bigl(K^{t-s+i+1}(A)\bigr)_{\mathfrak p}
	\cong
	K^{i+1}_{A_{\mathfrak p}}(A_{\mathfrak p}).
	\]
	This proves the equivalence.
\end{proof}
We single out the codimension-one case, where the comparison map
\(\varphi_{\mathfrak p}\) detects the first possible failure of \(B\) beyond
the \(S_2\)-condition. Equivalently, this failure is measured by the
\(\mathfrak pA_{\mathfrak p}\)-torsion of
\(H^2_{IA_{\mathfrak p}}(A_{\mathfrak p})\).
 \begin{corollary}
 	Let $(A,\mathfrak m)$ be a local ring that is a homomorphic image of a Gorenstein local ring, and assume that $J(A)=0$. Let $B$ and $C$ be as in \eqref{exact_ABC}. Assume \(C\neq 0\), and let \(t=\dim C\).
 	
 	Let $\mathfrak p\in \Supp^{\mathrm{top}}(C)$ be such that
 	\[
 	\dim A/\mathfrak p = t-1,
 	\]
 	i.e.\ $\mathfrak p$ is a codimension-$1$ point of $\Supp^{\mathrm{top}}(C)$. Then the following conditions are equivalent:
 	\begin{enumerate}
 		\item[(a)] $\varphi_{\mathfrak p}:(K^{t+1}(A))_{\mathfrak p}\to (K^t(C))_{\mathfrak p}$ is an isomorphism;
 		\item[(b)] $\depth B_{\mathfrak p}\ge 3$;
 		\item[(c)] $H^0_{\mathfrak p A_{\mathfrak p}}\!\big(H^2_{I A_{\mathfrak p}}(A_{\mathfrak p})\big)=0$;
 		\item[(d)] $J(K^{t+1}(A))_{\mathfrak p}=0$.
 	\end{enumerate}
 \end{corollary}
 
\begin{proof}
Here \(s=1\). The equivalence of \textup{(a)} and \textup{(b)} follows from
Proposition~\ref{prop:stratified-comparison}: in (b), the comparison for \(i=0\) is
automatic from the \(S_2\)-condition on \(B_{\mathfrak p}\), see
\eqref{B_p_S_2}, while the comparison for \(i=1\), after Matlis duality, is
precisely the localization of \(\varphi\) at \(\mathfrak p\).
	
	The equivalence of \textup{(a)} and \textup{(d)} follows from
	\(\ker\varphi=J(K^{t+1}(A))\), together with the fact that
	\(\operatorname{coker}\varphi\) has dimension at most \(t-2\), and hence
	vanishes after localizing at \(\mathfrak p\).
	
	It remains to compare \textup{(a)} and \textup{(c)}. Consider the localized
	Grothendieck spectral sequence
	\[
	E_2^{a,b}
	=
	H^a_{\mathfrak p A_{\mathfrak p}}
	\bigl(H^b_{I A_{\mathfrak p}}(A_{\mathfrak p})\bigr)
	\Longrightarrow
	H^{a+b}_{\mathfrak p A_{\mathfrak p}}(A_{\mathfrak p}).
	\]
	We have
	\[
	H^0_{I A_{\mathfrak p}}(A_{\mathfrak p})=0,
	\qquad
	H^1_{I A_{\mathfrak p}}(A_{\mathfrak p})\cong C_{\mathfrak p}.
	\]
	Since \(\dim C_{\mathfrak p}=1\), one has
	\(H^2_{\mathfrak p A_{\mathfrak p}}(C_{\mathfrak p})=0\). Hence the terms of
	total degree \(2\) give a short exact sequence
	\[
	0 \to H^1_{\mathfrak p A_{\mathfrak p}}(C_{\mathfrak p})
	\to H^2_{\mathfrak p A_{\mathfrak p}}(A_{\mathfrak p})
	\to
	H^0_{\mathfrak p A_{\mathfrak p}}
	\bigl(H^2_{I A_{\mathfrak p}}(A_{\mathfrak p})\bigr)
	\to 0.
	\]
	By Matlis duality, \(\varphi_{\mathfrak p}\) is an isomorphism if and only if
	\[
	H^1_{\mathfrak p A_{\mathfrak p}}(C_{\mathfrak p})
	\longrightarrow
	H^2_{\mathfrak p A_{\mathfrak p}}(A_{\mathfrak p})
	\]
	is an isomorphism. By the short exact sequence above, this is equivalent to
	\[
	H^0_{\mathfrak p A_{\mathfrak p}}
	\bigl(H^2_{I A_{\mathfrak p}}(A_{\mathfrak p})\bigr)=0.
	\]
	Thus \textup{(a)} and \textup{(c)} are equivalent.
\end{proof}
\subsection{Examples and applications}

We end with examples illustrating several aspects of the preceding results.
The first class comes from codimension two lattice ideals, where the
Peeva--Sturmfels resolution gives a natural way to prove indecomposability of
the relevant deficiency module and hence to deduce connectedness of the
top-dimensional non-\(S_2\) locus. We then give a flexible fiber-product construction producing
\(S_2\)-ifications with prescribed \(S_k\)-behavior. Finally,
we include a small example showing that the defect module \(C=B/A\) need not be
a quotient ring of \(A\).

\subsubsection{Codimension two lattice ideals}
Let $S = k[x_1,\dots,x_n]$ be a polynomial ring over a field $k$. For a vector $\underline{u} \in \mathbb{N}^n$, denote the corresponding monomial in $S$ by
\(x^{\underline{u}} := x_1^{u_1}\cdots x_n^{u_n}.\)
Let $\mathcal{L} \subseteq \mathbb{Z}^n$ be a sublattice and set \[
\Gamma:= \mathbb{Z}^n/\mathcal{L}.
\]  The associated lattice ideal in $S$ is defined by
\[
I_\mathcal{L}:= \langle x^{\underline{u}} - x^{\underline{v}} : \underline{u}, \underline{v} \in \mathbb{N}^n,\ \underline{u}-\underline{v}\in \mathcal L \rangle.
\]
The codimension of $I_\mathcal{L}$ equals the rank of $\mathcal{L}$. When $I_\mathcal{L}$ is prime, the quotient ring $S/I_\mathcal{L}$ is called an affine semigroup ring (or a toric ring), and $I_\mathcal{L}$ is said to be toric.
The ideal $I_\mathcal{L}$ is prime if and only if the lattice $\mathcal{L}$ is saturated, that is,
\begin{equation}\label{eqsat}
	\mathcal{L} = \mathcal{L}^{\mathrm{sat}} := \{\, \underline{u} \in \mathbb{Z}^n : r\underline{u}\in \mathcal L \text{ for some } r \in \mathbb{Z}_{>0} \,\},
\end{equation}
see \cite{peeva} and \cite[Chapter~7]{semigroup}.

	Before stating the theorem, we note that 
	by \cite[Prop.~4.1]{peeva}, the  assumption that $I_\mathcal{L}$ has at least $4$ minimal
	generators is equivalent to assuming that $A:=S/I_{\mathcal L}$ is not Cohen--Macaulay.
	
	In this subsection we use the preceding local results in their standard
	graded-local form.
\begin{theorem}\label{lattice_ideal_theorem}
	Let $I_{\mathcal L}$ be a lattice ideal of codimension $2$ minimally  generated by at least $4$ elements, and set $A := S/I_{\mathcal L}$. Let $d := \dim A = n-2$. Then $K^{d-1}(A)$ is $\Gamma$-graded indecomposable.
	
	If the non-\(S_2\) locus of \(A\) is nonempty, then it
	has dimension \(d-2\). 
	Moreover, suppose that $I_{\mathcal L}$ is prime and   that \(K^{d-1}(A)\) is equidimensional and
 \(S_2\), then   the top-dimensional part of the non-\(S_2\)
 locus,
 \(	\operatorname{non}\text{-}S_2^{\mathrm{top}}(A)\),  is connected in codimension \(1\). 
\end{theorem}

\begin{proof}
	We first prove the graded indecomposability of \(K^{d-1}(A)\). The argument is
	combinatorial and uses the structure of the Peeva--Sturmfels resolution.
Let $m \ge 4$ be the minimal number of generators of $I_{\mathcal L}$. By \cite[Chapter~5]{peeva}, Peeva and Sturmfels construct a minimal $\Gamma$-graded free resolution of $S/I_{\mathcal L}$ of the form
\[
0 \to F_3 \to F_2 \to F_1 \to S \to S/I_{\mathcal L} \to 0,
\]
where $F_3$ is indexed by the $m-3$ syzygy quadrangles, and the complex is obtained by gluing the quadrangle resolutions $F_{\mathcal C}$ (see \cite[Construction~5.2]{peeva}) along a homology tree $\mathcal{T}_{\mathcal L}$ (as defined in \cite[p.~175]{peeva}), which is a finite directed tree by \cite[Thm.~4.5]{peeva}.

Since \(K^{d-1}(A) \cong \operatorname{Ext}^{3}_S(A,S) \cong \operatorname{coker}(d_3^*)\), it suffices to show that $M := \operatorname{coker}(d_3^*)$ is $\Gamma$-graded indecomposable.

Write \[
F_3 = \bigoplus_{Q} S(-\gamma_Q), \qquad
F_2 = \bigoplus_{T} S(-\tau_T),
\] where \(Q\) runs over the syzygy quadrangles and \(T\) runs over the
syzygy triangles. By \cite[Comments~5.9(c)]{peeva}, each multigraded
Betti number \(\beta_{i,\gamma}\), \(\gamma\in \Gamma\), is either \(0\)
or \(1\). In particular, the degrees \(\gamma_Q\) are pairwise distinct,
and the degrees \(\tau_T\) are pairwise distinct.

We first show that all degree-$0$ $\Gamma$-graded endomorphisms of $F_3$ and $F_2$ are diagonal.

We use the following consequence of the proof of
\cite[Lem.~5.6]{peeva}. Let \(1\le i\le 2\), and let
\(C,D\in \Gamma \) be degrees of minimal
\((i+1)\)-st syzygies. If \(C\ne D\), then neither \(D-C\) nor \(C-D\)
is the degree of a nonconstant monomial in \(S\); equivalently,
\begin{equation}\label{consequence}
S_{D-C}=S_{C-D}=0 .
\end{equation}
Indeed, if \(D=C+\deg y\) for some nonconstant monomial \(y\), then the
argument in the proof of \cite[Lem.~5.6]{peeva} shows that
\(\Delta_D\) is contractible, contradicting the fact that \(D\) is the
degree of a minimal \((i+1)\)-st syzygy.

Now fix two distinct quadrangles \(Q\ne Q'\). A degree-\(0\)
\(\Gamma\)-graded map
\(S(-\gamma_Q) \to S(-\gamma_{Q'})\)
is given by multiplication by an element of
\(S_{\gamma_Q-\gamma_{Q'}}\). Since \(\gamma_Q\) and \(\gamma_{Q'}\) are
distinct degrees of minimal third syzygies, the preceding argument  
\eqref{consequence} gives
\(S_{\gamma_Q-\gamma_{Q'}}=0 .\)
Thus there are no off-diagonal degree-\(0\) maps between distinct summands
of \(F_3\). Since  the multigraded Betti
numbers are all \(0\) or \(1\), every degree-\(0\) \(\Gamma\)-graded
endomorphism of \(F_3\) is diagonal.
The same argument applied to the degrees \(\tau_T\) of the minimal second
syzygies shows that every degree-\(0\) \(\Gamma\)-graded endomorphism of
\(F_2\) is diagonal. The same conclusion holds for the dual modules \(F_3^*\) and
\(F_2^*\).

We now show that \(M\)
is \(\Gamma\)-graded indecomposable.
Suppose for contradiction that $M \cong M' \oplus M''$ is a nontrivial decomposition in the category of \(\Gamma\)-graded
\(A\)-modules. Then there exists a minimal $\Gamma$-graded presentation of $M$ that is block diagonal:
\[
G_1' \oplus G_1'' \xrightarrow{\psi' \oplus \psi''} G_0' \oplus G_0'' \to M \to 0.
\]
On the other hand,
\[
F_2^* \xrightarrow{d_3^*} F_3^* \to M \to 0
\]
is the minimal \(\Gamma\)-graded presentation obtained from the
Peeva--Sturmfels resolution.  Hence these two minimal presentations differ by degree-$0$ $\Gamma$-graded automorphisms of $F_2^*$ and $F_3^*$. Since these automorphisms are diagonal, passing from $d_3^*$ to $\psi' \oplus \psi''$ only rescales rows and columns.
It follows that the matrix of \(d_3^*\) would have to admit a nontrivial block
decomposition after partitioning the basis elements of \(F_2^*\) and \(F_3^*\)
into two nonempty blocks.

Let \(\mathcal G\) be the bipartite graph whose
vertices are the basis elements of \(F_3\) and \(F_2\), equivalently the
syzygy quadrangles \(Q\) and the syzygy triangles \(T\), and where there is
an edge \(Q-T\) if and only if the coefficient of the basis element \(e_T\)
in \(d_3(e_Q)\) is nonzero.  

The connectedness of the homology tree \(\mathcal T_{\mathcal L}\)
\cite[Thm.~4.5]{peeva} implies the connectedness of the graph
\(\mathcal G\). We spell this out in terms of the entries of \(d_3\).
By \cite[Construction~5.1]{peeva}, there is an edge between a quadrangle
\(Q\) and a triangle \(T\) precisely when \(T\) is one of the four
triangles belonging to \(Q\).

Let $P_1, \dots, P_{m-3}$ be an ordering of the syzygy quadrangles compatible with the homology tree $\mathcal{T}_{\mathcal{L}}$, for example as in \cite[Cor.~4.7]{peeva}. By
\cite[Lem.~5.7]{peeva}, the root quadrangle \(P_1\) gives four minimal
second syzygies, while each \(P_i\), \(i\ge 2\), contributes exactly two
additional minimal second syzygies. Thus, for each \(i\ge 2\), among the
four triangles belonging to \(P_i\), two have already appeared for some
earlier quadrangle \(P_j\), \(j<i\).

It follows that the graph \(\mathcal G\) is connected.
Since \(d_3^*\) is represented by the transpose matrix of \(d_3\), the same
bipartite graph also describes the nonzero entries of \(d_3^*\). But a block diagonal presentation  would make the graph $\mathcal{G}$   disconnected, because there would be no nonzero matrix entries between the two nonempty blocks, contradicting the connectedness of $\mathcal{G}$.

Therefore, $M$ is $\Gamma$-graded indecomposable.

We now prove the second part of the theorem. 
By the Peeva--Sturmfels minimal resolution, we have $\operatorname{pd}_S A = 3$. Together with   Auslander--Buchsbaum, this gives $\depth A = \dim S - \operatorname{pd}_S A = n - 3 = d - 1.$

Let
\(0\to A\to B\to C\to 0\)
be the \(S_2\)-ification sequence. Suppose that the non-\(S_2\) locus of \(A\) is nonempty, equivalently
\(C\neq 0\). We have
\(\dim C\le d-2\), and now show that \(\dim C=d-2\). 

Suppose, to the contrary, that \(\dim C=d-s<d-2\), so \(s\ge 3\). Since
\(\depth A=d-1\), we have 
\(K^i(A)=0\)  for \(i\le d-2\).
From the long exact sequence of deficiency modules, we obtain an isomorphism
\[
K^{d-s}(C)\cong K^{d-s}(B).
\]
But \(K^{d-s}(C)\) is the canonical module of \(C\), so
\[
\dim K^{d-s}(C)=\dim C=d-s.
\]
On the other hand, since \(B\) is \(S_2\), Proposition~\ref{prop_canonical_module}(c)
gives
\[
\dim K^{d-s}(B)\le d-s-2,
\]
a contradiction. Hence \(\dim C=d-2\), that is, the non-$S_2$ locus which is defined by $\ann _A C$ has dimension $d-2$.

We now assume that \(I_{\mathcal L}\) is prime  and that \(K^{d-1}(A)\)  is equidimensional and \(S_2\).
 By Proposition~\ref{prop:assh-equality}, 
\[
\Supp  K^{d-1}(A) =\Supp^{\mathrm{top}} K^{d-1}(A)
=
\Supp^{\mathrm{top}} K^{d-2}(C)
=
\operatorname{non}\text{-}S_2^{\mathrm{top}}(A).
\]
When \(I_{\mathcal L}\) is prime, $\Gamma$ is a torsion-free abelian group. Hence Theorem~\ref{indecomposable_graded}   applied to the \(\Gamma\)-graded
indecomposable  module \(K^{d-1}(A)\) shows that this top-dimensional
support is connected in codimension \(1\). This finishes the proof. 
\end{proof}

\begin{remark}
	The connectedness conclusion above does not   follow directly from
	Grothendieck's connectedness theorem. Assume \(J(A)=0\), and let
	\(\mathfrak a\supset I_{\mathcal L}\) be the ideal
	defining the top-dimensional part of the non-\(S_2\) locus of \(A\). Then
	\(	\operatorname{ht} (\mathfrak a/I_{\mathcal L} )\ge 2.\)
	Let \(\widehat S\) denote the completion of \(S\) at the graded maximal
	ideal. Grothendieck's connectedness theorem \cite[Thm.~19.2.10]{local cohomology} gives
	\[
		c\bigl( S/\mathfrak a  S\bigr)\ge 
	c\bigl(\widehat S/\mathfrak a\widehat S\bigr)
	\ge
	\min\{c(\widehat S),\,\operatorname{sdim}\widehat S-1\}
	-\operatorname{ara} (\mathfrak a\widehat S ).
	\]
	Since
	\(	\operatorname{ara} (\mathfrak a\widehat S )
	\ge
	\operatorname{ht} (\mathfrak a )
	\ge 4,\)
	the right-hand side is at most \(n-5=d-3\). Thus this bound obtained from
	Grothendieck's theorem alone does not imply connectedness in codimension \(1\)
	of the non-\(S_2\) locus. The connectedness result above therefore uses
	additional information carried by the deficiency module \(K^{d-1}(A)\).
\end{remark}
For a local ring \(A\) with a canonical module \(\omega_A\), the \(S_2\)-ification
\(\operatorname{End}_A(\omega_A)\) is Cohen--Macaulay precisely when
\(\omega_A\) is Cohen--Macaulay; this is  standard from \cite[Prop.~2.2]{S_2-ification_universal}.
\begin{corollary}\label{cor_toric}
Let \(I_{\mathcal L}\) be a toric ideal of codimension \(2\), minimally
generated by at least \(4\) elements, and set \(A:=S/I_{\mathcal L}\). Suppose
that  the canonical module \(\omega_A\) of \(A\) is  Cohen--Macaulay; for example, this holds when \(A\) is a simplicial affine semigroup ring by  \cite[Thm.~6.4]{ccm}. Then
the non-\(S_2\) locus of \(A\) coincides with the non-Cohen--Macaulay locus of
\(A\), and its top-dimensional part is connected in codimension \(1\).
\end{corollary}
\begin{proof}
Since \(B\) is Cohen--Macaulay,   the non-\(S_2\) locus agrees with the non-Cohen--Macaulay locus, and we have an isomorphism  \(K^{d-1}(A)\cong K^{d-2}(C)\). Thus  \(K^{d-1}(A)\) is equidimensional and $S_2$, the connectedness of its top-dimensional part then follows from Theorem~\ref{lattice_ideal_theorem}.
\end{proof}
The same connectedness-in-codimension-\(1\) conclusion also applies in certain
cases where \(I_{\mathcal L}\) is not necessarily prime. In the non-prime
case, the graded indecomposability result cannot be applied directly, since the
grading group \(\Gamma\) is not torsion-free. When \(I_{\mathcal L}\) is
minimally generated by exactly four elements, however, the Peeva--Sturmfels
resolution is sufficiently explicit to show directly that \(K^{d-1}(A)\) is
indecomposable as an ordinary \(A\)-module. The nongraded
indecomposability--connectedness criterion then gives the following
connectedness-in-codimension-\(1\) result.
\begin{proposition}\label{cor_lattice}
	Let $I_{\mathcal L}$ be a lattice ideal of
	codimension \(2\), minimally generated by \(4\) elements. Suppose that $A := S/I_{\mathcal L}$ is equidimensional and unmixed. Then the non-\(S_2\) locus of \(A\) is nonempty and has dimension
	\(\dim A-2\). Moreover,  $\operatorname{non}\text{-}S_2^{\mathrm{top}}(A)$ is connected in
	codimension \(1\), and  the non-\(S_2\) locus of \(A\) coincides
	with the non-Cohen--Macaulay locus.
\end{proposition}

\begin{proof}
	Let $d =\dim A$. 
By \cite[Construction~5.2]{peeva}, \(A\) has a minimal
free resolution 
\begin{equation}\label{resolution}
	0\to S\xrightarrow{d_3} S^4\xrightarrow{d_2} S^4\xrightarrow{d_1} S
	\to A\to 0,
\end{equation}
where the differential map $d_3 : S \to S^4$ is given by
\[
d_3 : S \xrightarrow{
	\begin{pmatrix}
		- x^{\underline{s}} \\
		x^{\underline{t}} \\
		x^{\underline{r}} \\
		- x^{\underline{p}}
	\end{pmatrix}
} S^4,
\]
with $x^{\underline{s}}, x^{\underline{t}}, x^{\underline{r}}, x^{\underline{p}}$  pairwise relatively prime monomials. Hence
\(K^{d  -1}(A)
\cong
S/(x^{\underline{s}},x^{\underline{t}},x^{\underline{r}},x^{\underline{p}}).\)
Since these monomials  form a regular
sequence in \(S\), \(K^{d-1}(A)\) is Cohen--Macaulay of dimension $d-2$. Hence \(A\) cannot be
\(S_2\); otherwise Proposition \ref{prop_canonical_module}(c) would give
\(\dim K^{d-1}(A)\le d-3\). Thus \(C\neq 0\), and Theorem \ref{lattice_ideal_theorem} gives
\(\dim C=d-2\) and $\depth A=d-1$.

We now prove that \(K^{d-1}(A)\) is indecomposable as an \(A\)-module. Since
\(K^{d-1}(A)\) is an \(A\)-module, the ideal \(I_{\mathcal L}\) annihilates it.
 Hence
\(I_{\mathcal L}\subseteq
(x^{\underline{s}},x^{\underline{t}},
x^{\underline{r}},x^{\underline{p}}) \) and  \(K^{d-1}(A)\cong A/J\) for some   ideal \(J\subseteq A\). Thus
\[
\operatorname{End}_A(K^{d-1}(A))
\cong \operatorname{End}_A(A/J)
\cong A/J
\cong
S/(x^{\underline{s}},x^{\underline{t}},
x^{\underline{r}},x^{\underline{p}}).
\]
The last ring is naturally \(\mathbb N\)-graded with degree-zero part \(k\), and
therefore has no nontrivial idempotents. Hence \(K^{d-1}(A)\) is
indecomposable as an \(A\)-module. Since \(K^{d-1}(A)\) is Cohen--Macaulay, it
is \(S_2\) and equidimensional. Therefore, by
Theorem~\ref{sheaf_indecomposable}, \(\Supp  K^{d-1}(A)  =   \operatorname{non}\text{-}S_2^{\mathrm{top}}(A) \)
  is connected in codimension $1$. 
  
 We now  show that the non-\(S_2\) locus coincides with the
  non-Cohen--Macaulay locus. 
 Since \(K^{d+1}(A)=K^{d+2}(A)=0\), equivalently
 \(\Ext_S^0(A,S)=\Ext_S^1(A,S)=0\), the dualized complex of \eqref{resolution} is exact at its first two
 terms. 
 Therefore we have exact sequences
 \[
 0\to S\xrightarrow{d_1^*}S^4\to \operatorname{im} d_2^*\to 0,
 \]
 \[
 0\to \ker d_3^*\to S^4\to \operatorname{im} d_3^*\to 0,
 \]
 \[
 0\to \operatorname{im} d_3^*\to S\to K^{d-1}(A)\to 0,
 \]
 \[
 0\to \operatorname{im} d_2^*\to \ker d_3^*\to \omega_A\to 0.
 \]
 Since \(K^{d-1}(A)\) is Cohen--Macaulay of dimension \(d-2\), the third exact
 sequence gives \(\depth \operatorname{im} d_3^*=d-1.\) The first two exact sequences give \(\depth \operatorname{im} d_2^*\ge d+1 \) and \(\depth \ker d_3^*\ge d.\) Thus the last exact sequence gives \(\depth \omega_A\ge d.\)
 Since \(\dim \omega_A=d\), the canonical module \(\omega_A\) is
 Cohen--Macaulay, by \cite[Prop.~2.2]{S_2-ification_universal},   
 \(B=\End_A(\omega_A)\) is also Cohen--Macaulay.
  Therefore
  the non-\(S_2\) locus of \(A\) coincides with its non-Cohen--Macaulay locus.

 \end{proof}

\subsubsection{\(S_2\)-ifications that are \(S_k\) but not \(S_{k+1}\)}
The next construction gives a family of examples whose \(S_2\)-ifications are
explicit and satisfy \(S_k\) but not \(S_{k+1}\). The input is a
Stanley--Reisner ring that is \(S_k\) but not \(S_{k+1}\).

 Let \(\Delta\) be a simplicial complex over a field \(K\). For a face
 \(F\in \Delta\), recall that the link of $F$ in $\Delta $ is given by 
 \[
 \operatorname{lk}_\Delta(F)
 =
 \{\,G\in\Delta : F\cap G=\emptyset,\; F\cup G\in \Delta\,\}.
 \]
 By Hochster's formula and the homological characterization of Serre's
 condition, the Stanley--Reisner ring \(K[\Delta]\) satisfies \(S_r\) if and
 only if, for every face \(F\in\Delta\),
 \begin{equation}\label{S_r condition}
 	\widetilde H_i(\operatorname{lk}_\Delta(F);K)=0
 	\qquad
 	\text{for } i<\min\{r-1,\dim\operatorname{lk}_\Delta(F)\};
 \end{equation}
 see  \cite[p.~1]{Murai--Terai}.

Fix an integer \(k\ge 2\), and let \(\Delta_k\) be a finite triangulation of \(S^{k-1}\times S^{k-1}\). Set
\(\widetilde R_k:=K[\Delta_k].\)
Since \(\dim\Delta_k=2k-2\), we have
\(\dim \widetilde R_k=2k-1.\) 

For the empty face,
\(\operatorname{lk}_{\Delta_k}(\emptyset)=\Delta_k\). By the Künneth formula,
\[
\widetilde H_i(S^{k-1}\times S^{k-1};K)=0
\quad\text{for } i<k-1,
\qquad
\widetilde H_{k-1}(S^{k-1}\times S^{k-1};K)\ne 0.
\]
Thus the condition \eqref{S_r condition} holds for \(F=\emptyset\) when
\(r=k\), but fails for \(r=k+1\)  already in degree \(i=k-1\).

Since \(|\Delta_k|\) is a closed
topological manifold of dimension \(2k-2\), by
\cite[Thm.~63.2]{homology_sphere}, for every nonempty face
\(\emptyset \neq F\in\Delta_k\), the link \(\operatorname{lk}_{\Delta_k}(F)\) has the
homology of a sphere of dimension
\((2k-2)-\dim F-1=2k-2-|F|\).
In particular, its reduced homology vanishes in all degrees below its top
dimension. Hence \eqref{S_r condition} is satisfied for every nonempty face
\(F\) when \(r=k\). Therefore \(\widetilde R_k\) satisfies \(S_k\)  but not \(S_{k+1}\).

Let
\(R_k:=(\widetilde R_k)_{\mathfrak n},\)
where \(\mathfrak n\) is the graded maximal ideal. Then
\(\dim R_k=2k-1,\)
and \(R_k\) is \(S_k\)

Choose an \(R_k\)-regular sequence \(\ell_1,\ell_2\in \mathfrak n\), and set
\[
J:=(\ell_1,\ell_2),\qquad B_k:=R_k\times R_k.
\]
Define
\[
A_k:=R_k\times_{R_k/J}R_k
=
\{(r_1,r_2)\in R_k\times R_k: r_1\equiv r_2 \pmod J\}.
\]
Then we have an exact sequence
\[
0\to A_k\to B_k\to C_k\to 0,
\qquad C_k\cong R_k/J.
\]
Moreover, \(C_k\) is cyclic as an \(A_k\)-module; indeed,
\(C_k\cong A_k/\operatorname{ann}_{A_k}C_k\).
Thus \(B_k\) is a finite \(A_k\)-module  and we have 
\(\dim A_k=\dim B_k=\dim R_k\). Since \(J\) is generated by an \(R_k\)-regular sequence of length \(2\),
\[
\dim C_k=\dim R_k-2=\dim A_k-2.
\]
Furthermore, \(B_k=R_k\times R_k\) is \(S_k\), hence \(S_2\). Therefore, by the standard
characterization of the \(S_2\)-ification \cite[Thm.~1.6]{S_2-ification_universal}, \(B_k\) is the
\(S_2\)-ification of \(A_k\).

For completeness, we also record the local cohomology of this family, which makes explicit how the comparison map \(\varphi\) behaves in these examples. In this case, since \(C_k\) is a quotient ring of \(A_k\), the canonical module \(\omega_{C_k}\) of \(C_k\) agrees with the canonical module \(K_{C_k}\) of \(C_k\) viewed as an \(A_k\)-module.

Let \(\mathfrak n\) denote the maximal ideal of \(R_k\), equivalently, the graded maximal ideal of \(\widetilde{R_k}\), and let
\(\mathfrak m\) denote the maximal ideal of \(A_k\). By Hochster's formula
\cite[Thm.~5.3.8]{cohen macaulay}, the multigraded pieces   \(\bigl(H^i_{\mathfrak n}(\widetilde R_k)\bigr)_a \) vanish unless 
\(a\in \mathbb Z^N_{\le 0},\)
where \(N\) is the number of vertices of \(\Delta_k\). For such \(a\), set
\( G_a:=\{j:a_j<0\}.\) If \(G_a\notin\Delta_k\), then
\(\bigl(H^i_{\mathfrak n}(\widetilde R_k)\bigr)_a=0\). If
\(G_a\in\Delta_k\), then Hochster's formula  \cite[Thm.~5.3.8]{cohen macaulay} gives
\[
\bigl(H^i_{\mathfrak n}(\widetilde R_k)\bigr)_a
\cong
\widetilde H_{i-|G_a|-1}\bigl(\operatorname{lk}_{\Delta_k}(G_a);K\bigr).
\]
  Suppose \(G_a\ne\emptyset\), then   \(\operatorname{lk}_{\Delta_k}(G_a)\) has the
homology of a sphere of dimension
\(2k-2-|G_a|\). Thus
\(\widetilde H_{i-|G_a|-1}\bigl(\operatorname{lk}_{\Delta_k}(G_a);K\bigr)\ne 0\)
can happen only when
\(i-|G_a|-1=2k-2-|G_a|,\)
that is, only when \(i=2k-1\). Set \(d:=2k-1=\dim R_k\). It follows that, for \(i<d\), the only possible
nonzero multigraded piece of \(H^i_{\mathfrak n}(\widetilde R_k)\) occurs when
\(G_a=\emptyset\). 
  Hence,   we obtain \(K\)-vector space isomorphisms
\[
H^i_{\mathfrak n}(  R_k)\cong
H^i_{\mathfrak n}(\widetilde R_k)
\cong
\bigl(H^i_{\mathfrak n}(\widetilde R_k)\bigr)_0
\cong
\widetilde H_{i-1}(\Delta_k;K)
\qquad (i<d).
\]
Set \(T_k:=H^{2k-1}_{\mathfrak n}(R_k),\)
and using that 
\[
\widetilde H_j(S^{k-1}\times S^{k-1};K)
=
\begin{cases}
	K^{\oplus 2}, & j=k-1,\\
	K, & j=2k-2,\\
	0, & \text{otherwise},
\end{cases}
\]
we have 
\begin{equation}\label{R_k_local_cohomology}
	H^i_{\mathfrak n}(R_k)\cong
	\begin{cases}
		K^{\oplus 2}, & i=k,\\
		T_k, & i=2k-1,\\
		0, & \text{otherwise}.
	\end{cases}
\end{equation}
Thus
\begin{equation}\label{B_k_local_cohomology}
	H^i_{\mathfrak m}(B_k)
	\cong
	H^i_{\mathfrak n}(R_k)\oplus H^i_{\mathfrak n}(R_k)
	\cong
	\begin{cases}
		K^{\oplus 4}, & i = k,\\
		T_k^{\oplus 2}, & i = 2k - 1,\\
		0, & \text{otherwise}.
	\end{cases}
\end{equation}
Now \(C_k\cong R_k/J\), where \(J\) is generated by part of an
\(R_k\)-regular sequence. Using the long exact sequences of local cohomology
for \(R_k\to R_k/J\) and for \(0\to A_k\to B_k\to C_k\to 0\), one computes,
for \(k\ge 4\),
\[
H^i_{\mathfrak m}(C_k)\cong
\begin{cases}
	K^{\oplus 2}, & i=k-2,\\
	K^{\oplus 4}, & i=k-1,\\
	K^{\oplus 2}, & i=k,\\
	(0:J)_{T_k}, & i=2k-3,\\
	0, & \text{otherwise},
\end{cases}
\]
and
\[
H^i_{\mathfrak m}(A_k)\cong
\begin{cases}
	K^{\oplus 2}, & i=k-1,\\
	K^{\oplus 6}, & i=k,\\
	(0:J)_{T_k}, & i=2k-2,\\
	T_k^{\oplus 2}, & i=2k-1,\\
	0, & \text{otherwise}.
\end{cases}
\]
In this case the comparison map
\(\varphi\colon K^{2k-2}(A_k)\longrightarrow K^{2k-3}(C_k)\)
is an isomorphism.

For \(k\ge 4\), since \(R_k\) is \(S_k\) and \(J\) is generated by an
\(R_k\)-regular sequence of length \(2\), it follows that
\(C_k\cong R_k/J \)
is \(S_{k-2}\), hence \(S_2\). Moreover, \(C_k\) is equidimensional and
unmixed. Therefore
\(\operatorname{End}_{C_k}(\omega_{C_k})\cong C_k.\)
Since \(C_k\) is local, \(\omega_{C_k}\) is indecomposable as a
\(C_k\)-module. Consequently, the
non-\(S_2\) locus of \(A_k\), equivalently \(\Spec C_k\), is connected in
codimension \(1\).

When \(k=3\), \eqref{R_k_local_cohomology} and
\eqref{B_k_local_cohomology} still hold, but the long exact sequences for
\(A_3\) and \(C_3\) do not split as cleanly. In this case,
\[
H^i_{\mathfrak m}(C_3)\cong
\begin{cases}
	K^{\oplus 2}, & i=1,\\[4pt]
	K^{\oplus 4}, & i=2,\\[4pt]
	\text{an extension }
	0\to K^{\oplus 2}\to H^3_{\mathfrak m}(C_3)
	\to (0 : J)_{T_3}\to 0, & i=3,\\[4pt]
	0, & \text{otherwise}.
\end{cases}
\]
Similarly, from the long exact sequence associated to
\(0\to A_3\to B_3\to C_3\to 0\), one obtains
\[
H^i_{\mathfrak m}(A_3) \cong
\begin{cases}
	K^{\oplus 2}, & i = 2,\\[4pt]
	\text{an extension } 
	0 \to K^{\oplus 4} \to H^3_{\mathfrak m}(A_3) \to K^{\oplus 2}\to 0, & i = 3,\\[4pt]
	(0 : J)_{T_3}   , & i = 4,\\[4pt]
	T_3	 ^{\oplus 2}, & i = 5,\\[4pt]
	0, & \text{otherwise}.
\end{cases}
\]
In this case, the map
\(\varphi : K^{4}(A_3)\to K^{3}(C_3)\)
fits into a short exact sequence
\begin{equation}\label{exact_sequence_k3}
	0 \longrightarrow K^{4}(A_3)
	\xrightarrow{\ \varphi\ }
	K^{3}(C_3)
	\longrightarrow K^{\oplus 2}
	\longrightarrow 0.
\end{equation}

To show that $C_3$ is connected in codimension $1$, we apply Grothendieck’s connectedness theorem \cite[Thm.~19.2.10]{local cohomology}. Since \(R_3\) is \(S_2\), its completion \(\widehat{R_3}\) is also \(S_2\). Thus \(\widehat{R_3}\) is connected in
codimension \(1\), so
\[
c(\widehat{R_3}) \ge \dim \widehat{R_3} - 1 = 4.
\]
Moreover, since \(R_3\) is equidimensional and universally catenary,
\(\widehat{R_3}\) is equidimensional; see, for example, the proof of
Corollary~\ref{universally catenary}. Thus
\[
\operatorname{sdim}\widehat{R_3}
=
\dim\widehat{R_3}
=
\dim R_3
=
5.
\]
 Applying Grothendieck's connectedness
theorem to the ideal \(J\widehat{R_3}\), we get
\[
c( {C_3})\geq c(\widehat{C_3})
=
c(\widehat{R_3}/J\widehat{R_3})
\ge
\min\{c(\widehat{R_3}),\operatorname{sdim}\widehat{R_3}-1\}
-
\operatorname{ara}(J\widehat{R_3})
\ge 2.
\]
Since \(\dim C_3=3\), it follows that \( \operatorname{Spec }C_3\) is connected in codimension
\(1\). Consequently, by Theorem \ref{canonical_M} \(K^3(C_3) \cong \omega_{C_3}\) is indecomposable.
 
\begin{remark}
	Indeed, since \(\Delta_k\) is a triangulation of the closed orientable manifold
	\(S^{k-1}\times S^{k-1}\), Gräbe's description \cite[\S3.5]{grabe} of canonical modules of
	Stanley--Reisner rings of homology manifolds gives
	\(	\omega_{\widetilde R_k}\cong \widetilde R_k,\)
	hence \(\omega_{R_k}\cong R_k\). Therefore, for \(k\ge 4\),
	\[
	\omega_{C_k}\cong \omega_{R_k}/J\omega_{R_k}\cong R_k/J\cong C_k.
	\]
	For $k = 3$, the sequence \eqref{exact_sequence_k3} becomes
	\[
	0 \to C_3 \xrightarrow{i} \omega_{C_3} \to K^{\oplus 2} \to 0.
	\]
\end{remark}

\subsubsection{A noncyclic \(S_2\)-defect module}

Let \((E,\mathfrak n)\) be a Cohen--Macaulay local ring with canonical module
\(\omega_E\), and assume that \(E\) is not Gorenstein. Then \(\omega_E\) is
noncyclic; see \cite[Ch.~3]{local cohomology}.

Define
\[
D:=E\ltimes \omega_E,
\qquad
B:=D[[u,v]],\qquad A:=
\left\{
\sum_{i,j\ge 0}(a_{ij},m_{ij})u^iv^j\in B
\;\middle|\;
m_{00}=0
\right\}.
\]
The ring \(A\) is local. 
The map \(\rho\colon A\to E\), 
\(\rho\!\left(\sum_{i,j\ge 0}(a_{ij},m_{ij})u^iv^j\right)=a_{00}\)
makes \(\omega_E\) an \(A\)-module. And the map \[
\pi:B\to \omega_E,\qquad
\pi\!\left(\sum_{i,j\ge 0}(a_{ij},m_{ij})u^iv^j\right)=m_{00}
\]
is \(A\)-linear and has kernel \(A\). Hence, setting \(C:=B/A\), we have an
exact sequence of \(A\)-modules
\[
0\to A\to B\to C\to 0,
\qquad
C\cong \omega_E.
\]
Moreover, by \cite[Cor.~2.12]{aoyama}, \(B\) is Gorenstein, hence \(S_2\). Since the \(A\)-action on
\(C\cong\omega_E\) factors through \(E\), we have
\[
\dim_A C =\dim E=\dim B-2.
\]
Also,   \(C\cong \omega_E\) is finite over \(E\),  thus \(C\) is finite over \(A\). Then  \(B\) is also finite over \(A\),  hence  $A$ is noetherian by \cite[Thm.~I.3.7]{matsumura}, and \(\dim A=\dim B\).
By the standard characterization of \(S_2\)-ifications \cite[Thm.~1.6]{S_2-ification_universal}, \(B\) is the
\(S_2\)-ification of \(A\).

Since \(\omega_E\) is noncyclic, \(C\) is not of the form \(A/I\) for any
ideal \(I\subseteq A\). Thus, in contrast to the preceding fiber-product
examples, the defect module \(C\) need not be a quotient ring defining the
non-\(S_2\) locus; it may occur only as a finite \(A\)-module.

	\appendix

	\bibliographystyle{alpha}

\end{document}